\documentclass[reqno,a4paper,11pt,oneside,final]{amsart}   
\usepackage{etex}
\usepackage{float}
\usepackage{mathrsfs}
\usepackage{amssymb,amsmath,amsthm,amsfonts}
\usepackage{stmaryrd} % Für (ll|rr)bracket
\usepackage{mathtools}
\usepackage{enumitem}
\usepackage{xcolor}
\usepackage{booktabs}
\usepackage[english]{babel}
\usepackage[T1]{fontenc}
\usepackage{lmodern}

\usepackage[final,stretch=10]{microtype}

\usepackage[colorlinks, citecolor=hanblue,linkcolor=red]{hyperref}

\usepackage[
	style=authoryear-comp,
	sorting=nyt,
	url=true,
	doi=true,                % Print the DOI.
	isbn=false,              % Do not print ISBN/ISSN
	eprint=true,
	uniquename=init,         % Disable a warning
	maxbibnames=99,          % Do not use ``et al'' in the bibliography -- now: only with 99+ authors ;)
	maxcitenames=3,
	uniquelist=minyear,
	dashed=false,            % Do not use ---- to replace the authors in the bibliography (if the previous entry has the same authors)
	sortcites=false,          % Do not sort if multiple keys in one \cite{}
	backend=biber,           % Use biber as backend
	safeinputenc,            % For some accents, see http://tex.stackexchange.com/questions/170562
	date=year,               % Do not show months in the References
]{biblatex}
\DeclareCiteCommand{\parencite}[\mkbibbrackets]{%
	\ifbibmacroundef{cite:init}{}{\usebibmacro{cite:init}}\usebibmacro{prenote}%
}{%
	\usebibmacro{citeindex}%
	\printtext[bibhyperref]{\usebibmacro{cite}}%
}{%
	\ifbibmacroundef{cite:init}{\multicitedelim}{}%
}{%
	\usebibmacro{postnote}%
}%

\let\cite\parencite

\usepackage[noabbrev,nameinlink,capitalize]{cleveref}
\definecolor{hanblue}{rgb}{0.27, 0.42, 0.81}
\usepackage{todonotes}

\DeclarePairedDelimiter\parens()
\DeclarePairedDelimiter\bracks[]

\DeclarePairedDelimiter{\abs}{\lvert}{\rvert}
\DeclarePairedDelimiter{\norm}{\lVert}{\rVert}

\DeclarePairedDelimiter{\seq}()

\providecommand\given{\nonscript\;\delimsize|\nonscript\;\mathopen{}}
\DeclarePairedDelimiterX\set[1]\{\}{#1}

\DeclarePairedDelimiterX\dual[2]{\langle}{\rangle}{#1,#2}
\DeclarePairedDelimiterX\dualM[2]{\langle}{\rangle_C}{#1,#2}
\DeclarePairedDelimiterX\ddual[2]{\llangle}{\rrangle}{#1,#2}
\DeclarePairedDelimiterX\eqclass[2]{\llbracket}{\rrbracket}{#1,#2}

\DeclareMathOperator{\dom}{dom}
\DeclareMathOperator{\Id}{Id}
\DeclareMathOperator{\supp}{supp}

\DeclareMathOperator{\Ran}{Ran}

\DeclareMathOperator{\cl}{cl}
\DeclareMathOperator{\diam}{diam}
\DeclareMathOperator{\lip}{lip}
\let\Im\Ran

\newcommand{\TangentCone}{\mathcal{T}}

\newcommand{\N}{\mathbb{N}}
\newcommand{\R}{\mathbb{R}}
\newcommand{\A}{\mathcal{A}}
\newcommand{\I}{\mathcal{I}}
\newcommand{\MM}{\mathcal{M}}
\newcommand{\BB}{\mathcal{B}}
\newcommand\de{\mathop{}\!\mathrm{d}}
\newcommand{\M}{\mathcal{M}(\Omega)}
\newcommand{\C}{C(\Omega)}
\newcommand{\Co}{C^1(\Omega)}
\newcommand{\Lip}{\operatorname{Lip}(\Omega)}
\newcommand{\Md}{\mathcal{M}(\Omega;\R^d)}
\newcommand{\Cd}{C(\Omega;\R^d)}
\newcommand{\Cod}{\Co^*}
\newcommand{\Moc}{\M}

\DeclarePairedDelimiter{\mnorm}  {\lVert}{\rVert_{\mathcal{M}}}
\DeclarePairedDelimiter{\cnorm}  {\lVert}{\rVert_{C}}
\DeclarePairedDelimiter{\conorm} {\lVert}{\rVert_{C^1}}
\DeclarePairedDelimiter{\codnorm}{\lVert}{\rVert_{(C^1)^*}}
\DeclarePairedDelimiter{\mdnorm} {\lVert}{\rVert_{\mathcal{M}^d}}
\DeclarePairedDelimiter{\cdnorm} {\lVert}{\rVert_{C^d}}
\DeclarePairedDelimiter{\lnorm}  {\lVert}{\rVert_{\operatorname{Lip}}}
\DeclarePairedDelimiter{\blnorm} {\lVert}{\rVert_{\operatorname{BL}}}
\DeclarePairedDelimiter{\ynorm}  {\lVert}{\rVert_{Y}}

\newcommand{\Kb}{\mathcal{K}}

\newcommand{\eps}{\varepsilon}
\newcommand{\Intr}{\operatorname{int}}

\newcommand{\sgn}{\operatorname{sgn}}
\newcommand{\weakstar}{\stackrel{*}{\rightharpoonup}}  %

\newcommand\loss{L}

\DeclareFontFamily{OMX}{MnSymbolE}{}
\DeclareSymbolFont{MnLargeSymbols}{OMX}{MnSymbolE}{m}{n}
\SetSymbolFont{MnLargeSymbols}{bold}{OMX}{MnSymbolE}{b}{n}
\DeclareFontShape{OMX}{MnSymbolE}{m}{n}{
    <-6>  MnSymbolE5
   <6-7>  MnSymbolE6
   <7-8>  MnSymbolE7
   <8-9>  MnSymbolE8
   <9-10> MnSymbolE9
  <10-12> MnSymbolE10
  <12->   MnSymbolE12
}{}
\DeclareFontShape{OMX}{MnSymbolE}{b}{n}{
    <-6>  MnSymbolE-Bold5
   <6-7>  MnSymbolE-Bold6
   <7-8>  MnSymbolE-Bold7
   <8-9>  MnSymbolE-Bold8
   <9-10> MnSymbolE-Bold9
  <10-12> MnSymbolE-Bold10
  <12->   MnSymbolE-Bold12
}{}
 
\makeatletter
\let\llangle\@undefined
\let\rrangle\@undefined
\makeatother
\DeclareMathDelimiter{\llangle}{\mathopen}%
                     {MnLargeSymbols}{'164}{MnLargeSymbols}{'164}
\DeclareMathDelimiter{\rrangle}{\mathclose}%
                     {MnLargeSymbols}{'171}{MnLargeSymbols}{'171}

\numberwithin{equation}{section}
 
\definecolor{darkred}{rgb}{.7,0,0}

\definecolor{green}{rgb}{0,0.7,0}

\usepackage[notref,notcite]{showkeys}

\theoremstyle{plain}

\newtheorem{theorem}{Theorem}[section]
\newtheorem{lemma}[theorem]{Lemma}
\newtheorem{corollary}[theorem]{Corollary}
\newtheorem{proposition}[theorem]{Proposition}

\theoremstyle{definition} 

\newtheorem{assumption}[theorem]{Assumption}
\newtheorem{definition}[theorem]{Definition}
\newtheorem{example}[theorem]{Example}
\newtheorem{remark}[theorem]{Remark}

\begin{document}
\title[Second-order conditions for Radon norm regularization]{No-gap second-order conditions for minimization problems in spaces of measures}

\pagestyle{myheadings}

 \author[G. Wachsmuth]{Gerd Wachsmuth}
\author[D. Walter]{Daniel Walter}

 \address[Gerd Wachsmuth]{Institute of Mathematics, Brandenburgische Technische Universität Cottbus–Senftenberg, 03046 Cottbus, Germany}

\address[Daniel Walter]{Institut f\"ur Mathematik, Humboldt-Universit\"at zu Berlin, 10117 Berlin, Germany}

\email[Gerd Wachsmuth]{gerd.wachsmuth@b-tu.de}
 \email[Daniel Walter]{daniel.walter@hu-berlin.de}

\begin{abstract}
\small{ % 
Over the last years, minimization problems over spaces of measures have received increased interest due to their relevance in the context of inverse problems, optimal control and machine learning. A fundamental role in their numerical analysis is played by the assumption that the optimal dual state admits finitely many global extrema and satisfies a second-order sufficient optimality condition in each one of them.         
 In this work, we show the full equivalence of these structural assumptions to a no-gap second-order condition involving the second subderivative of the Radon norm as well as to a local quadratic growth property of the objective functional with respect to the bounded Lipschitz norm.

 
}
\end{abstract}
 
\maketitle
\section{Introduction} \label{sec:intro}
 In this paper, second-order necessary and sufficient optimality conditions for sparse minimization problems of the form 
 \begin{align} \label{def:sparseintro}
 \min_{u} J(u) \coloneqq \loss(Ku)+ \alpha \mnorm{u} \tag{$P$}
 \end{align}
 are studied.
Here,~$\mnorm{\cdot}$ denotes the canonical norm on the space of Radon measures~$\M$ on a compact spatial domain~$\Omega$,~$\alpha>0$,
$\loss \colon Y \to \R $
is a smooth, but not necessarily convex, loss function
on a Hilbert space $Y$ of observations
and~$K \colon \M \to Y$ is a linear and continuous control-to-observation operator given by 
 \begin{align} \label{def:integraloperator}
     Ku= \int_\Omega k(x) \de u(x) \qquad \forall u \in \M
 \end{align}
 for some sufficiently smooth~$k \colon \Omega \to Y  $.
 Problems of this or similar form appear in a variety of challenging settings in the context of inverse problems,~\cite{scherzer,pikkarainen}, such as source identification in acoustics,~\cite{Trautmannseismic,pieper_tang_trautmann_walter_2020} or microscopy,~\cite{McCutchen:67,denoyelle_duval_peyre_2020}, optimal control,~\cite{clasonphoto,casas,piepervexler}, or the training of shallow neural networks,~\cite{bach}, and optimal sensor placement,~\cite{neitzel_2019,kiefer}. This is largely attributed to the observation that the appearance of the nonsmooth Radon norm in the objective functional promotes \textit{sparse minimizers}, i.e., solutions given as a finite linear combination of Dirac measures
\begin{align} \label{eq:sparsemeasureintro}
	\bar{u}= \sum^N_{j=1} \bar{\lambda}_j \delta_{\bar{x}_j} \quad \text{with} \quad \Bar{x}_j \in \Omega,~\Bar{\lambda}_j \in \R \setminus \set{0},
\end{align}
where $\delta_{\Bar x_j} \in \M$ is the Dirac measure in the point $\Bar x_j$,
i.e.,
$\int_\Omega \varphi \de \delta_{\Bar{x}_j}=\varphi(\Bar{x}_j)$
holds for all continuous functions~$\varphi$ on~$\Omega$.
From an analytic perspective, the sparsity promoting property of the Radon norm follows, e.g., by convex representer theorems, if the image of~$K$ is finite dimensional, \cite{Boyer19,Carioni20}, or it can be deduced from the first-order necessary optimality condition for the Jordan decomposition~$\Bar{u}= \Bar{u}_+ -\Bar{u}_{-}$,
\begin{align} \label{eq:firstorderintro}
    \supp \Bar{u}_{\pm} \subset \set*{ x \in \Omega \given \Bar{p}(x)= \pm \alpha} \quad  \text{as well as} \quad |\Bar{p}(x)| \leq \alpha \quad \text{on}~ \Omega
\end{align}
provided that the \textit{dual variable}~$\Bar{p}=-(k(\cdot), \nabla \loss(K \Bar{u}))_Y \in C(\Omega)$ admits finitely many extrema, see~\cite[Proposition 3.8]{pieper_walter_2021}.

Given a sparse control~$\Bar{u}$,
represented as in \eqref{eq:sparsemeasureintro},
with associated dual variable~$\Bar{p}$ satisfying
\begin{align} \label{eq:strictcompintro}
    \supp \Bar{u} = \set*{ x \in \Omega \given |\Bar{p}(x)|=  \alpha}= \{\Bar{x}_j\}^N_{j=1} \subset \operatorname{int}(\Omega),
\end{align}
we call~$\Bar{p}$~\textit{non-degenerate} if we have
\begin{align} \label{eq:nondegintro}
	\operatorname{dim}\parens*{ \operatorname{span} \set*{k(\Bar{x}_1),\dots,k(\Bar{x}_N) } }=N
	\quad \text{and} \quad
	\sgn(\Bar{\lambda}_j)\nabla^2 \Bar{p}(\Bar{x}_j) \leq_L -\theta \operatorname{Id} \quad\forall j = 1,\ldots, N,
\end{align}
where $\theta > 0$ and
``$\leq_L$'' denotes the Loewner order. Non-degeneracy plays a fundamental role in the analysis of sparse minimization problems allowing, e.g., a fine analysis of the asymptotic regularization properties of the Radon norm,~\cite{duval_peyre_2014}, the derivation of sharp a priori error estimates for discretizations of~\eqref{def:sparseintro}, or the proof of fast convergence rates for problem-tailored numerical solution algorithms,~\cite{chizat,pieper_walter_2021,flinth}.

\subsection{Contribution} \label{subsec:contribution} 
Given a sparse measures~$\bar{u}$, i.e. a measure in the form~\eqref{eq:sparsemeasureintro}, which satisfies the first-order necessary optimality condition~\eqref{eq:firstorderintro} together with its associated dual variable~$\Bar{p}$, we investigate~\textit{no-gap} second-order conditions for the sparse minimization problem~\eqref{def:sparseintro}. By the latter, we refer to the property that positive semidefiniteness of a suitable second-order derivative at~$\Bar{u}$ is a necessary condition for its minimality while positive definiteness implies strict local minimality as well as a local quadratic growth behavior of the objective functional~$J$ around~$\Bar{u}$. For this purpose, an equivalent ``lifted'' version of Problem~\eqref{def:sparseintro} onto~$\Cod$, the topological dual space of the continuous differentiable functions, is considered,
\begin{align*} 
    \min_{h \in \Cod} \loss(\Kb(h))+G(h),
\end{align*}
involving suitable extensions~$\mathcal{K} \colon \Cod \to Y  $ as well as~$G \colon \Cod \to (-\infty, \infty]$ which satisfy
\begin{align*}
    \mathcal{K}(\imath_{\mathcal{M}}u)=Ku, \quad G(\imath_{\mathcal{M}} u)= \alpha \mnorm{u} \quad \forall u \in \M \quad \text{and} \quad G(h)= \infty \quad \forall h \in   \Cod \setminus \M,
\end{align*}
where~$\imath_{\mathcal{M}}$ denotes the compact embedding of~$\M$ into~$\Cod$. Our abstract main result, \cref{thm:main}, clarifies the connection between structural properties of the dual variable~$\Bar{p}$, in particular its curvature around~$\supp \Bar{u}$, the second order condition 
\begin{align*}
(\Kb h,\nabla^2 \loss(K \Bar{u}) \Kb h)_Y   + G''( \imath_{\mathcal{M}} \bar u, \Bar{p}; h) >0 \qquad\forall h \in \Cod \setminus \{0\},
\end{align*}
involving~$G''( \imath_{\mathcal{M}} \bar u, \Bar{p}; \cdot) \colon \Cod \to [-\infty,\infty]$, the \emph{weak* second subderivative} of~$G$ at~$\imath_{\mathcal{M}} \bar u$ for~$\Bar{p}$,
as well as local quadratic growth of the objective functional in the vicinity of~$\Bar{u}$ w.r.t the bounded Lipschitz norm 
\begin{equation*}
\blnorm{u}= \sup_{\lnorm{\varphi}\leq 1} \int_\Omega \varphi \de u \quad \text{where} \quad 	\lnorm{\varphi} \coloneqq \max\set{\cnorm{\varphi},\lip(\varphi)},
\end{equation*}
and~$\lip(\varphi)$ denotes the Lipschitz constant. While we emphasize that \cref{thm:main} covers nonconvex loss functions~$L$ and contains no restriction to settings in which~$\Bar{p}$ only admits a finite number of global extrema, our general analysis reveals non-degeneracy in the sense of~\eqref{eq:nondegintro} as a special case.
Indeed, \cref{coroll:nondegfromssc,coroll:nondegfromssc}
shows the equivalency of the following statements
under the assumption that ~$L$ is strongly convex and that the strict complementarity condition~\eqref{eq:strictcompintro} holds.
\begin{enumerate}
\item The dual variable~$\Bar{p}$ is non-degenerate, i.e., it satisfies~\eqref{eq:nondegintro}.
\item  There holds
\begin{align}
(\Kb h,\nabla^2 \loss(K \Bar{u}) \Kb h)_Y   + G''( \imath_{\mathcal{M}} \bar u, \Bar{p}; h) >0 \qquad\forall h \in \Cod \setminus \{0\},
\end{align}
and
\begin{equation*}
		\begin{aligned}
			&\text{for all $\seq{t_k} \subset (0,\infty)$, $\seq{h_k} \subset \Cod$
			with $t_k \searrow 0$, $h_k \weakstar 0$}
			\text{ and $\codnorm{h_k} = 1$, we have }\\
			&\qquad
			\liminf_{k \to \infty} \parens[\bigg]{
				\frac{1}{t^2_k}\parens*{G(\imath_{\mathcal{M}}\Bar{u}+t_k h_k)-G(\imath_{\mathcal{M}}\Bar{u})}
				-
				\frac{1}{t_k}\ddual{\Bar{p}}{ h_k}
			}
			> 0		
		\end{aligned}
\end{equation*}
where~$\ddual{\cdot}{ \cdot}$ denotes the duality pairing between~$\Co$ and~$\Cod$.
\item  There exist~$\gamma>0$ and $\varepsilon >0$ such that there holds
\begin{align} \label{eq:quadintro}
	J(u)-J(\Bar{u}) \geq \gamma \blnorm{u-\bar{u}}^2 \qquad\forall u \in \M, \blnorm{u-\bar{u}} \leq \varepsilon. 
\end{align}
\end{enumerate}
A main ingredient of our derivation of \cref{thm:main} is an explicit formula for ~$G''( \imath_{\mathcal{M}} \bar u, \Bar{p}; \cdot)$ which we obtain as a byproduct of showing the weak* twice epi-differentiable of~$G$. Moreover, we also conclude that the sparsity of~$\bar{u}$ is a fundamental requirement for our results: Indeed, see \cref{lem:ndc_implies_isolated_support_2}, quadratic growth w.r.t the bounded Lipschitz norm around a measure~$\Bar{u}$ can only hold if~$\Bar{u}$ is a finite sum of Dirac measures.
\subsection{Related work} \label{subsec:related}
The present work is mainly influenced by prior work on no-gap second-order conditions utilizing the notion of second subderivatives.
This is a well-known concept in finite-dimensional variational analysis,
see the classical treatment in \cite[Definition~13.3]{RockafellarWets1998}
as well as the recent contribution \cite{BenkoMehlitz2022}. In the infinite-dimensional setting,
this derivative has been used in \cite{Do1992}
to characterize the directional differentiability of projections.
More recently, \cite{ChristofWachsmuth2017:1},
second subderivatives have been utilized to derive second-order conditions for infinite-dimensional constrained minimization problems.
Subsequently, this approach has been further developed
in the contributions
\cite{wachsmuth2,BorchardWachsmuth2023}.

In the present work, we follow an implicit approach to computing $G''( \imath_{\mathcal{M}} \bar u, \Bar{p}; \cdot)$ in order to avoid the necessity of working on the dual space~$\Cod$. More in detail, we first prove the strong-strong twice epi-differentiability of the preadjoint of~$G$, which is now a functional over~$\Co$, and derive a candidate for $G''( \imath_{\mathcal{M}} \bar u, \Bar{p}; \cdot)$ via arguments from convex duality. Conceptually, this is closest related to the exposition in~\cite{wachsmuth2} where the authors consider a similar treatment for integral functionals defined on~$L^1(\Omega)$.   
Finally, we also mention 
\cite{ChristofWachsmuth2017:3}
where the authors address the differentiability of
parameter-to-solution maps associated to parameter-dependent variational inequalities by using second subderivatives.

The notion of non-degeneracy of certain dual variables was possibly first coined in the seminal paper~\cite{duval_peyre_2014} where it was used to address the recovery properties of Radon norm variational regularization in inverse problems. Similar applications can be found in, e.g.,~\cite{poon_2018,poon_2023,huynh2023optimal}. Subsequently, these structural assumptions have also become a central pillar for the analysis of finite element discretizations of PDE-constrained optimal control problems with sparse controls, see, e.g.,~\cite{vexler}, as well as the derivation of fast convergence results for numerical algorithms such as generalized conditional gradient or exchange type methods,~\cite{pieper_walter_2021,flinth}, or over-parametrized gradient descent,~\cite{chizat}. In this context, we also point out that~\cite[Proposition 3.1]{chizat} deduces local quadratic growth w.r.t the bounded Lipschitz norm from a slightly stronger analogue of~\eqref{eq:strictcompintro} and~\eqref{eq:nondegintro}.

If~$Y$ is finite dimensional and~$L$ is convex, we also note that Problem~\eqref{def:sparseintro} can be interpreted as the Fenchel dual of a semi-infinite, state-constrained problem. From this perspective,~$\bar{u}$ corresponds to the Lagrange multiplier for the state constraint imposed on the optimal state~$\Bar{p}$. Moreover,~\eqref{eq:strictcompintro} can be interpreted as a strict complementarity condition. In this context, assumptions on the curvature of~$\bar{p}$ in the vicinity of~$\supp \Bar{u}$, have been utilized, e.g., for the derivation of approximation error estimates,~\cite{neitzelsemi}.
\subsection{Outline} \label{subsec:outline}
After introducing the necessary notation in \cref{sec:notation}, we collect some basic but important results on both, sparse minimization problems and second-order conditions based on second subderivatives, in \cref{sec:sparseoptistart,sec:nogapabstract}, respectively. Subsequently, the latter are applied to Problem~\eqref{def:sparseintro}. Our main abstract result, \cref{thm:main}, as well as tangible conclusions based upon it are stated in \cref{sec:nogapmeasures}. The following sections, \cref{sec:impliesSSC,sec:strucimplyNDC}, are dedicated to its technical proof. Finally, \cref{app:quasi} contains some required auxiliary results used throughout the paper.   
\section{Notation} \label{sec:notation}
Let $d$ be a positive integer.
In the following, let~$\Omega \subset \R^d$ be compact and assume that $\Omega$ equals the closure of its interior. By $\BB(\Omega)$ we denote the Borel $\sigma$-algebra of $\Omega$. We further introduce the space of Radon measures~$\M$ on~$\Omega$ as the topological dual space of~$\C$ which is equipped with the maximum norm
\begin{align*}
    	\cnorm{\varphi} \coloneqq \max\set{ \abs{\varphi(x)} \given x \in \Omega } \quad \forall \varphi \in \C.
\end{align*}
The corresponding duality pairing is given by
\begin{align*}
    \dual{\varphi}{u}= \int_\Omega \varphi \de u \quad \forall \varphi \in \C, u \in \M.
\end{align*}
We equip~$\M$ with the canonical dual norm
\begin{align*}
    \mnorm{u} \coloneqq \max\set{ \dual{\varphi}{u}\given \cnorm{\varphi} \leq 1 } \quad \forall u \in \M
\end{align*}
making it a Banach space.

For a function $\varphi \colon \Omega \to \R$, we define its Lipschitz constant
\begin{equation}
	\label{eq:lipschitz_constant}
	\lip(\varphi) \coloneqq \sup\set*{\frac{\varphi(x_1)-\varphi(x_2)}{\abs{x_1-x_2}} \given x_1,x_2 \in \Omega, x_1 \ne x_2}.
\end{equation}
In particular, $\varphi$ is Lipschitz continuous if and only if $\lip(\varphi) < \infty$.

The space $\Co$ consists of all functions
$\varphi \in \C$ which are continuously differentiable on $\Intr(\Omega)$
such that $\nabla \varphi$ can be continuously extended from $\Intr(\Omega)$
to $\Omega$.
Under the canonical norm
\begin{align*}
	\conorm{\varphi} = \max \set{ \cnorm{\varphi},\cdnorm{\nabla \varphi} }
	\qquad \text{where} \qquad  \cdnorm{\varphi} \coloneqq \sup\set{ \abs{\varphi(x)} \given x \in \Omega } \qquad  \forall \varphi \in \Co,
\end{align*}
the space $\Co$ becomes a Banach space, see \cite[Section~1.6]{Alt2011}.
Here, $\abs{\cdot}$ is the Euclidean norm on $\R^d$.

For all $x \in \R^d$ and $r > 0$, we denote by
\begin{equation*}
	B_r(x)
	:=
	\set{
		y \in \R^d
		\given
		\abs{x - y} \le r
	}
\end{equation*}
the closed ball around $x$ with radius $r$.

Finally, we abbreviate $A v^2 \coloneqq v^\top A v$ for $A \in \R^{d \times d} $ and~$v \in \R^d$.
Similarly,
if $b \colon X \times X \to \R$ is a bilinear mapping,
we abbreviate $b h^2 := b[h,h]$.
This will be used, in particular, for second derivatives.

\section{Minimization problems in spaces of measures} \label{sec:sparseoptistart}
In this section, we collect some preliminary results on minimization problems of the form~\eqref{def:sparseintro}. Where possible, proofs are omitted for the sake of brevity.
\begin{assumption} \label{ass:functions}
    Throughout the paper, we assume that:
    \begin{enumerate}[label=\textbf{A\arabic*}]
        \item The space of observations~$Y$ is a separable Hilbert space with norm~$\ynorm{\cdot}=\sqrt{(\cdot,\cdot)_Y}$.
        \item The functional~$\loss \colon Y \to \R  $ is of class~$\mathcal{C}^2$.
        \item The kernel~$k \colon \Omega \to Y$ satisfies~$k \in C^{1, \gamma}(\Omega; Y)$ for some~$0< \gamma \leq 1$.
        \item The set $\Omega \subset \R^d$ is compact,
          it equals the closure of its interior $\Intr(\Omega)$
          and $\Intr(\Omega)$ is uniformly locally quasiconvex, see \cref{def:LUQ}.
    \end{enumerate}
\end{assumption}
First, we note that the regularity of~$k$ implies weak*-to-strong continuity of the integral operator~$K$.
\begin{lemma} \label{lem:proposofsparseop}
    Let \cref{ass:functions} hold. Then the operator~$K \colon \M \to Y$ as defined in~\eqref{def:integraloperator} is linear and sequentially weak*-to-strong continuous. Moreover, we have~$K=(K_*)^*$ where the preadjoint~$K_* \colon Y \to \C$ is given by
    \begin{align*}
        \lbrack K_* y \rbrack (x) = (k(x),y)_Y \quad \forall y \in Y,~x \in \Omega. 
    \end{align*}
\end{lemma}
\begin{proof}
 Note that~$K_*$ is linear and bounded since~$k(x) \in C(\Omega;Y)$. Moreover, we observe
 \begin{align*}
     \dual{K_* y}{u}= \int_\Omega (k(x),y)_Y \de u(x)= \left(y, \int_\Omega k(x) \de u(x) \right)_Y=(y, Ku)_Y
 \end{align*}
 for all~$y \in Y$ and~$u \in \M$. Thus, we have~$(K_*)^*=K$. In particular,~$K$ is linear and bounded as well as sequentially weak*-to-weak continuous. Let~$\seq{u_k} \subset \M$ be a weak* convergent sequence with limit~$\Bar{u}$. Then there exists a sequence~$\seq{y_k} \subset Y$ with~$\ynorm{y_k}=1$,~$\ynorm{Ku_k}=(y_k, Ku_k)_Y$ and thus 
 \begin{align} \label{eq:lschelp}
     \ynorm{K\Bar{u}} \leq \liminf_{k \rightarrow \infty} \ynorm{Ku_k}= \liminf_{k \rightarrow \infty} (y_k, Ku_k)_Y = \liminf_{k \rightarrow \infty} \dual{K_* y_k}{u_k}.
 \end{align}
 Now, consider an arbitrary weakly convergent subsequence of~$\seq{y_k}$, denoted by the same symbol, with limit~$\Bar{y} \in Y$,~$\ynorm{\Bar{y}}\leq 1$. We readily verify that~$K_* y_k$ converges weakly to~$K_* \Bar{y}$ in~$\C$. Moreover, since~$k \in C^{0,\gamma}(\Omega;Y)$, the sequence $\seq{K_* y_k}$ is uniformly bounded in~$C^{0,\gamma}(\Omega)$. Hence, due to Arzel\`a--Ascoli theorem, we conclude~$K_* y_k \rightarrow K_* \Bar{y}$ in~$\C$ and, finally,
 \begin{align*}
     \lim_{k \rightarrow \infty} \dual{K_* y_k}{u_k}=\dual{K_* \Bar{y}}{\Bar{u}}\leq \ynorm{K\Bar{u}}.
 \end{align*}
 Together with~\eqref{eq:lschelp}, we arrive at~$\lim_{k \rightarrow \infty} \ynorm{Ku_k}= \ynorm{K\Bar{u}}$.
 A usual subsequence-subsequence argument yields $K u_k \to K \Bar u$ for the entire sequence,
 i.e.,~$K$ is weak*-to-strong continuous.
\end{proof}
In what follows, we denote by~$\alpha \partial \mnorm{u}$ the convex subdifferential of~$\alpha \mnorm{\cdot}$ at~$u \in \M$.
Next, we summarize necessary and sufficient first-order optimality conditions for minimizers of Problem~\eqref{def:sparseintro}. 
\begin{proposition} \label{prop:optimalitysparse}
    Given~$\Bar{u} \in \M$, define
    \begin{align*}
        \Bar{p}=- K_* \nabla \loss(K \bar{u})=- (k(\cdot), \nabla \loss(K \Bar{u}))_Y \in \Co.
    \end{align*}
    If the measure~$\Bar{u}$ is a minimizer of~\eqref{def:sparseintro},
    then the following equivalent conditions hold.
    \begin{enumerate}
        \item We have~$\bar{p} \in \alpha \partial \mnorm{\Bar{u}}$.
        \item There holds~$|\Bar{p}(x)|\leq \alpha$ for all $x \in \Omega$ as well as~$\dual{\Bar{p}}{\Bar{u}}=\alpha \mnorm{\Bar{u}}$.
        \item The Jordan decomposition~$\Bar{u}=\Bar{u}_+-\Bar{u}_{-}$ satisfies
        \begin{align*}
     \supp \Bar{u}_{\pm} \subset \set*{ x \in \Omega \given \Bar{p}(x)= \pm \alpha} \quad  \text{as well as} \quad |\Bar{p}| \leq \alpha \quad \text{on}~ \Omega
     \end{align*}
    \end{enumerate}
    If $\loss$ is convex and if these conditions hold,
    then $\bar u$ is a minimizer of~\eqref{def:sparseintro}.
\end{proposition}
\section{No-gap second-order conditions for abstract minimization problems}
\label{sec:nogapabstract}
Next, we summarize tangible results on no-gap second-order conditions for nonsmooth minimization problems. Subsequently, these will be specialized to the problem at hand. More in detail, we consider 
\begin{equation}
	\label{eq:abstract}
	\min_{x\in X}
	\left \lbrack F(x) + G(x) \right \rbrack
	.
\end{equation}

We start by fixing the setting.
\begin{assumption} \label{ass:abstract}
The following assumptions are made throughout this section.
\begin{enumerate}
	\item
		There holds~$X=Z^*$ for a separable Banach space~$Z$.
	\item
		The functional $G \colon X \to (-\infty,\infty]$
		and
		the point
		$\bar x \in \dom(G)$ are
		given.
	\item
		Associated with the functional~$F \colon \dom(G) \to \R$,
		there exist
		$F'(\bar x) \in Z$
		and
		a bounded bilinear form
		$F''(\bar x) \colon X \times X \to \R$
		such that
		\begin{equation}
			\label{eq:hadamard_taylor_expansion}
			\lim_{k \to \infty}
			\frac{F(\bar x + t_k h_k) - F(\bar x) - t_k F'(\bar x) h_k - \frac12 t_k^2 F''(\bar x) h_k^2}{t_k^2}
			=
			0
		\end{equation}
		holds
		for all sequences $\seq{t_k} \subset (0,\infty)$, $\seq{h_k} \subset X$
		satisfying $t_k \searrow 0$,  $h_k \weakstar h \in X$
		and
		$\bar x + t_k h_k \in \dom(G)$.
\end{enumerate}
\end{assumption}
Loosely speaking, the ``smooth'' part $F$ of the objective admits a second-order Taylor expansion around~$\bar{x}$.
The bilinear form~$F''(\bar{x})$ can be interpreted as its Hessian.
Regarding the potentially nonsmooth term~$G$, we rely on the notion of the weak* second subderivative.
\begin{definition} \label{def:weakssubderivative}
	Let~$x \in \dom(G)$ and~$p\in Z$ be given.
	The \emph{weak* second subderivative}
 \begin{align*}
     G''(x,p;\cdot)\colon X \to [-\infty,\infty]
 \end{align*}
	of~$G$ at~$x$ for~$p$ is given by
\begin{align*}
	G''(x,p;h) \coloneqq \inf\set*{
		\liminf_{k \rightarrow \infty} \frac{G(x+t_k h_k)-G(x)-t_k \dual{z}{h_k}}{t^2_k/2}
		\given
		t_k \searrow 0,~h_k \weakstar h
	}
	.
\end{align*}
The functional~$G$ is called \emph{strictly twice epi-differentiable} at~$x$ for~$p$
if for every~$h \in X$ and every sequence~$\seq{t_k}\subset(0,\infty)$ with $t_k \searrow 0$
there exists a sequence~$\seq{h_k} \subset X$ with
\begin{align*}
	G''(x,p;h)= \lim_{k \rightarrow \infty} \frac{G(x+t_k h_k)-G(x)-t_k \dual{z}{h_k}}{t^2_k/2},
	\quad h_k \weakstar h,~\norm{h_k}_X \to \norm{h}_X. 
\end{align*}
\end{definition}
We use the following result,
see \cite[Theorem~2.20]{BorchardWachsmuth2023}.
\begin{theorem} \label{thm:abstractSSC}
Let \cref{ass:abstract} hold and assume that~$h \mapsto F''(\bar{x})h^2 $ is sequentially weak* continuous.
Then the following assertions are equivalent.
\begin{enumerate}[label=(\roman*)]
	\item There exist~$\eps>0$ and~$\gamma >0$
		satisfying the second-order growth condition
		\begin{align*}
			(F+G)(x)-(F+G)(\bar{x}) \geq \frac{\gamma}{2} \norm{x-\bar{x}}^2_X
			\qquad \forall x \in X, \norm{x-\bar{x}}_X \leq \eps.
		\end{align*}
	\item We have
		the sufficient second-order condition
		\begin{align} \label{eq:SSC}
			\tag{SSC}
			F''(\bar{x})h^2+ G''(\bar{x},-F'(\bar{x});h) >0
			\qquad\forall h \in X \setminus \set{0}
		\end{align}
		as well as
		the non-degeneracy condition
	\begin{equation*}
		\label{eq:NDC}
		\tag{\textup{NDC}}
		\begin{aligned}
			&\text{for all $\seq{t_k} \subset (0,\infty)$, $\seq{h_k} \subset X$
			with $t_k \searrow 0$, $h_k \weakstar 0$}
			% \\[-0.1cm] &
			\text{ and $\norm{h_k}_X = 1$, we have }\\
			&\qquad
			\liminf_{k \to \infty} \parens[\bigg]{
				\frac1{t_k^2} \parens[\big]{G(\bar x + t_k h_k) - G(\bar x)}
				+
				\dual{F'(\bar x)}{h_k / t_k}
				+
				\frac12 F''(\bar x) h_k^2
			}
			> 0
			.
		\end{aligned}
	\end{equation*}
\end{enumerate}
\end{theorem}
We point out that this equivalency still holds under slightly weaker assumptions, see, e.g., the stated reference for more details.
From a practical perspective and for particular examples, it is desirable to link the abstract conditions~\eqref{eq:SSC} and~\eqref{eq:NDC} to tangible structural properties of the problem under consideration. The following auxiliary results will prove useful in this regard. Loosely speaking, these allow to circumvent the necessity to work on the potentially complicated dual space~$X$ by studying related properties of suitable pre-conjugates~$H \colon Z \to (-\infty, \infty] $,~$H^*=G$.    
For this purpose, we require the following strong analogue of \cref{def:weakssubderivative} for~$H$.
\begin{definition} \label{def:strong-strong}
Let~$p \in Z $ and~$x \in X$ be given. The \textit{strong second subderivative}
\begin{align*}
  H''(p,x;\cdot)\colon Z \to [-\infty,\infty]  
\end{align*}
of~$H$ at~$p$ for~$x \in X $ is given by
\begin{align*}
    H''(p,x;z)
    =
    \inf \set*{
        \liminf_{k \rightarrow \infty} \frac{H(p+t_k z_k)-H(p)-t_k \dual{z_k}{x}}{t^2_k/2}
        \given
        t_k \searrow 0,~z_k \rightarrow z~\text{in}~Z
    }.
\end{align*}
Accordingly, the functional~$H$ is called strongly-strongly twice epi-differentiable at~$p$ for~$x$
if for every~$z \in Z$ and every sequence~$\seq{t_k} \subset (0,\infty)$ with $t_k \searrow 0$,
there exists a sequence~$\seq{z_k} \subset Z$ with
\begin{align*}
H''(p,x;z)= \lim_{k \rightarrow \infty} \frac{H(p+t_k z_k)-H(p)-t_k \dual{z_k}{x}}{t^2_k/2}, \quad z_k \rightarrow z. 
\end{align*}
\end{definition}
We start by recalling that the convex conjugate of~$H''(p,x;\cdot)$ provides a lower bound on~$G''(x,p;\cdot)$,
see \cite[Lemma~4.4]{wachsmuth2}.
\begin{lemma} \label{lem:conjugatelowergen}
Assume that~$H \colon Z \to (-\infty,\infty]$ is proper, convex, lower semicontinuous and satisfies~$H^*=G$.
Moreover, let~$x \in \dom(G)$ and~$p \in Z$ with~$p \in \partial G(x)$ be given.
If $H$ is strongly-strongly twice epi-differentiable at $p$ for $x$, then there holds 
\begin{align*}
\frac12 G''(x,p; h) \geq \left( \frac12 H''(p,x;\cdot) \right)^*(h) \qquad\forall h \in X.
\end{align*}
\end{lemma}
Second, we provide a sufficient condition for~\eqref{eq:NDC} involving~$H$.
Note that this generalizes \cite[Lemma~4.9]{wachsmuth2}.
\begin{lemma}
    \label{lem:for_NDC}
    Assume that~$H \colon Z \to (-\infty,\infty]$ is proper, convex, lower semicontinuous and satisfies~$H^*=G$.
    Moreover, let~$\Bar{p} \in Z$ with~$\bar{p} \in \partial G(\Bar{x})$ as well as~$m\in\N$ be given.
    Fix some linearly independent $\zeta_1,\ldots,\zeta_m \in Z$ and define $H_2 \colon Z \to (-\infty, \infty]$
    via
    \begin{equation*}
        H_2\parens*{ \sum_{i = 1}^m \beta_i \zeta_i }
        =
        \frac12 \sum_{i = 1}^m \beta_i^2 + \dual*{\sum_{i = 1}^m \beta_i \zeta_i}{ \Bar{x}}
    \end{equation*}
    on~$\operatorname{span}\set{\zeta_1,\dots,\zeta_m}$
    and $H_2(z) = \infty$ outside of this linear hull.
    Let $\tilde H = H \oplus H_2$ denote the infimal convolution and suppose that
    \begin{equation}
        \label{eq:descent_lemma}
        \tilde H(p)
        \le
        H(\bar p)
        +
        \dual{p - \bar p}{\bar x}
        +
        \frac\Lambda2 \norm{p - \bar p}_Z^2
        \qquad\forall p \in Z, \norm{p - \bar p}_Z \le \eta
    \end{equation}
    holds for some $\Lambda, \eta > 0$.
    Then
    \begin{equation}
        \label{eq:ascent_lemma}
        G(x)
        \ge
        G(\bar x)
        +
        \dual{\bar p}{x - \bar x}
        +
        \frac1{2\Lambda} \norm{x - \bar x}_X^2
        -\frac12
        \sum_{i = 1}^m \abs{\dual{\zeta_i}{x - \bar x}}^2
        \qquad\forall x \in X, \norm{x - \bar x}_X \le \eta\Lambda.
    \end{equation}
    If further $X \ni h \mapsto F''(\bar x) h^2 \in \R$ is sequentially weak* lower semicontinuous
    and $\bar p = -F'(\bar x)$,
    then~\eqref{eq:NDC} holds.
\end{lemma}
\begin{proof}
    We have
    $(H \oplus H_2)^* = G + H_2^*$,
    see \cite[Exercise~4.5.7]{BorweinZhu2005}.
    The conjugate of $H_2$ is
    \begin{equation*}
        H_2^*(x)
        =
        \sup\set*{
            \dual*{\sum_{i = 1}^m \beta_i \zeta_i}{x}
            -
            \parens*{\frac12 \sum_{i = 1}^m \beta_i^2 + \dual*{\sum_{i = 1}^m \beta_i \zeta_i}{ \bar x}}
            \given \beta_i \in \R
        }
        =
        \frac12
        \sum_{i = 1}^m \abs{\dual{\zeta_i}{x - \bar x}}^2
        .
    \end{equation*}
    Thus, we have
    $0 \in \partial H_2^*(\bar x)$
    and
    $\bar p \in \partial (G + H_2^*)(\bar x) = \partial \tilde H^*(\bar x)$.
    Now, the Fenchel--Young identity gives
    \begin{equation*}
        \tilde H(\bar p)
        =
        \dual{\bar p}{\bar x}
        -
        (G + H_2^*)(\bar x)
        =
        \dual{\bar p}{\bar x}
        -
        G(\bar x)
        =
        H(\bar p).
    \end{equation*}
    Hence, we can replace $H(\bar p)$ in \eqref{eq:descent_lemma}
    by $\tilde H(\bar p)$.
    Next, we follow the proof of \cite[Lemma~4.9]{wachsmuth2} which yields
    \begin{equation*}
        \tilde H^*(x)
        \ge
        \tilde H^*(\bar x)
        +
        \dual{\bar p}{x - \bar x}
        +
        \frac1{2\Lambda} \norm{x - \bar x}_X^2
        \qquad\forall x \in X, \norm{x - \bar x}_X \le \eta\Lambda.
    \end{equation*}
    The identity $\tilde H^* = G + H_2^*$
    implies \eqref{eq:ascent_lemma}.
    Finally, let sequences $\seq{t_k} \subset (0,\infty)$, $\seq{h_k} \subset X$ with $t_k \searrow 0$,
    $h_k \weakstar 0$ in $X$ and $\norm{h_k}_X = 1$ be given.
    Then,
    we utilize \eqref{eq:ascent_lemma} to get
    \begin{align*}
        % \MoveEqLeft
        \liminf_{k \to \infty}
        \parens*{
            \frac{G(\bar x + t_k h_k) - G(\bar x) - \dual{\bar p}{t_k h_k}}{t_k^2}
            +
            \frac12 F''(\bar x) h_k^2
        }
        % \\
        &\ge
        \liminf_{k \to \infty}
        \parens*{
            \frac1{2\Lambda}
            -
            \frac12 \sum_{i = 1}^m \abs{\dual{\zeta_i}{h_k}}^2
            % +
            % \frac12 F''(\bar x) h_k^2
        }
        =
        \frac1{2\Lambda},
    \end{align*}
    where we used $h_k \weakstar 0$ and the assumed property of $F''(\bar x)$.
    This shows that \eqref{eq:NDC} is satisfied.
\end{proof}
\section{No-gap second-order conditions in spaces of measures} \label{sec:nogapmeasures}
In the remainder of the paper, we will use the auxiliary results of \cref{sec:nogapabstract} to investigate the abstract conditions~\eqref{eq:SSC} and~\eqref{eq:NDC} for a suitable interpretation of Problem~\eqref{def:sparseintro} and an
appropriate choice of the variable space. In this section, before formalizing our results, we therefore introduce an equivalent lifting of
Problem~\eqref{def:sparseintro} to~$X=\Cod$, the dual space of the continuous
differentiable functions on~$\Omega$.
 
\subsection{Uniform local Quasiconvexity} \label{subsec:quasiconvex}
In what follows,
we recall the notion of ``uniform local quasiconvexity''
appearing in \cref{ass:functions}
and discuss its consequences.
\begin{definition}
	\label{def:LUQ}
	We say that a set $D \subset \R^d$ is
	\emph{uniformly locally quasiconvex}
	with constants
	$r > 0$ and $C \ge 1$
	if for any
	$x,y \in D$ with $\abs{y - x} \le r$
	there exists a curve
	$\gamma \in C([0,1]; D)$
	with $\gamma(0) = x$, $\gamma(1) = y$
	and Lipschitz constant
	at most $C \abs{y - x}$.
\end{definition}

\begin{remark}
	\label{rem:rectifiable_vs_Lipschitz}
	Typically, it is required that the curve $\gamma \in C([0,1]; D)$
	from \cref{def:LUQ}
	is rectifiable and of length at most $C \abs{y - x}$
	(instead of being Lipschitz continuous),
	see \cite[p.~1223]{HajlaszKoskelaTuominen2008}.
	However, one can perform an arc-length reparametrization
	to ensure Lipschitz continuity of $\gamma$.
\end{remark}

It is clear that convex sets are uniformly locally quasiconvex.
Further,
the next result shows that \cref{def:LUQ}
is a rather weak requirement.
\begin{lemma}
	\label{lem:Lipschitz_implies_ulq}
	Let $D \subset \R^d$ be open and bounded.
	Further, we require that $D$
	is a Lipschitz set
	in the sense that
	for every point $p \in \partial D$,
	there exists $r > 0$ and
	a bijection $l_p \colon B_r(p) \to B_1(0)$
	such that $l_p$ and $l_p^{-1}$ are Lipschitz,
	\begin{equation*}
		l_p( D \cap B_r(p)) = \set{z \in B_1(0) \given z_n > 0} =: B_1^+(0)
		\quad\text{and}\quad
		l_p( \partial D \cap B_r(p)) = \set{z \in B_1(0) \given z_n = 0}
		.
	\end{equation*}
	Then,
	$D$ is uniformly locally quasiconvex.
\end{lemma}

The proof of \cref{lem:Lipschitz_implies_ulq}
can be found in \cref{app:quasi}.
We give two important consequences of \cref{def:LUQ}.
For the proofs, we again refer to \cref{app:quasi}.
The first lemma shows that functions from $\Co$
are Lipschitz continuous.
\begin{lemma}
	\label{lem:LUC_implies_C1_embeds_Lip}
	Let $\Intr(\Omega)$ be uniformly locally quasiconvex
	with constants $r > 0$ and $C \ge 1$.
	Then, every $\varphi \in \Co$
	is Lipschitz continuous and
	\begin{equation*}
		\lip(\varphi)
		\le
		\max\set*{
			\frac{2}{r} \cnorm{\varphi},
			C \cdnorm{\nabla\varphi}
		}
		\le
		\max\set*{2 r^{-1}, C}
		\conorm{\varphi}
		.
	\end{equation*}
\end{lemma}
The next lemma shows that we have a uniform Taylor expansion.
\begin{lemma}
	\label{lem:quasiconvex_taylor}
	Assume that $\Intr(\Omega)$ is uniformly locally quasiconvex and
	let $\varphi \in C^1(\Omega)$ be given.
	For any $\varepsilon > 0$,
	there exists $\delta > 0$ such that
	\begin{equation*}
		\abs*{
			\varphi(y) - \varphi(x) - \nabla \varphi(x)^\top (y - x)
		}
		\le
		\varepsilon \abs{y - x}
		\qquad
		\forall x, y \in \Omega, \abs{y - x} \le \delta.
	\end{equation*}
\end{lemma}

\begin{remark}
	\label{rem:generalization_possible}
	In case that $\Omega \subset \R^d$ is a domain,
	\cite[Theorem~7]{HajlaszKoskelaTuominen2008}
	suggests that
	the assertions of \cref{lem:LUC_implies_C1_embeds_Lip,lem:quasiconvex_taylor}
	hold
	if and only if
	$\Omega$ is uniformly locally quasiconvex. For compact sets $\Omega \subset \R^d$,
	we have seen that
	it is sufficient that $\Intr(\Omega)$
	is uniformly locally quasiconvex. It is a natural question whether it is also necessary in this context. This is indeed not the case:
	Consider $\Omega = [-1,0]^2 \cup [0,1]^2$.
	Clearly,
	$\Intr(\Omega) = (-1,0)^2 \cup (0,1)^2$
	is not uniformly locally quasiconvex.
	However,
	we can apply
	\cref{lem:LUC_implies_C1_embeds_Lip,lem:quasiconvex_taylor}
	on $[-1,0]^2$ and $[0,1]^2$
	separately
	and, consequently,
	the desired properties hold on $\Omega$. On the contrary, we provide another example in \cref{app:quasi}, \cref{ex:cantor}, showing that both conclusions, \cref{lem:LUC_implies_C1_embeds_Lip,lem:quasiconvex_taylor}, can fail in the absence of local uniform quasiconvexity. 
\end{remark}

\subsection{A primer on the dual space of~$\Co$ and the relation to the bounded Lipschitz norm} \label{subsec:dualofC1}
We start by giving a precise characterization of the space~$\Cod$ and of its canonical dual norm. For this purpose, note that
\begin{align*}
    \Co \cong C \coloneqq
    \set*{
        (\varphi,\nabla \varphi)
        \given
        \varphi \in \Co 
    }
    \subset \C \times \Cd,
\end{align*}
where the product space on the right-hand side is equipped with the norm
\begin{align*}
\norm{(\varphi,\phi)}_{C \times C^d}= \max \left\{\cnorm{\varphi},\cdnorm{\phi}\right\} \qquad \forall (\varphi,\phi) \in \C \times \Cd.
\end{align*}
Note that the mapping $\Co \ni \varphi \mapsto (\varphi,\nabla\varphi) \in \C \times \Cd$
is an isometry.
One readily verifies that the dual space
of $\C \times \Cd$
is given by~$\M \times \Md $ with associated duality pairing 
\begin{equation*}
\langle \varphi, u  \rangle + \langle \phi, v \rangle
\qquad \forall (\varphi,\phi) \in \C \times \Cd,~(u,v) \in \M \times \Md
\end{equation*}
and norm
\begin{equation*}
	\norm{(u,v)}_{\M \times \Md}
	=
	\mnorm{u} + \mdnorm{v}
	=
	\mnorm{u} + \mnorm{\abs{v}},
\end{equation*}
where $\abs{v} \in \M$ is the variation
of the vector-valued measure $v$
(w.r.t.\ the Euclidean norm on its image space $\R^d$).

The subset $C$ is complete since it is isometrically isomorphic
to the complete space $\Co$.
Thus,
$C$ is a closed subspace of~$\C \times \Cd $.
Consequently, its dual space can be characterized using the annihilator
\begin{align*}
    C^\bot
    \coloneqq
    \set*{
        (\mu,\nu) \in \M \times \Md
        \given
        \dual{\varphi}{\mu} + \dual{\nabla \varphi}{\nu} =0 \quad \forall(\varphi,\nabla \varphi) \in C
    }.
\end{align*}    
\begin{lemma}[{\cite[Proposition~11.10]{Brezis2011}}]
	\label{lem:identofdual}
There holds
\begin{equation}
\label{eq:identofdual}
(\M \times \Md)  / C^\bot \cong  \Cod.
\end{equation} 
The isometric isomorphism is given by
\begin{equation*}
	(\M \times \Md)  / C^\bot \ni
	(u,v) + C^\bot
	\mapsto
	\parens*{
		\Co \ni
		\varphi \mapsto
		\dual{\varphi}{u} + \dual{\nabla \varphi}{v}
	}.
\end{equation*}
\end{lemma}

Consequently, we identify the elements of~$\Cod$ with equivalence classes
\begin{equation*}
	\eqclass{u}{v} \coloneqq
	(u,v) + C^\bot
	=
	\set*{ (u,v)+(\mu,\nu) \given (\mu,\nu) \in C^\bot }
	\in
	(\M \times \Md)  / C^\bot
\end{equation*}
of measures~$(u,v) \in \M \times \Md$  
which act on functions~$\varphi \in \Co$ via the duality pairing
\begin{align*}
\ddual{\varphi}{\eqclass{u}{v}} \coloneqq \langle \varphi, u \rangle+ \langle \nabla \varphi, v \rangle .
\end{align*}
In what follows, we will always employ the identification \eqref{eq:identofdual},
i.e., bounded functionals on $\Co$ are always identified with the corresponding equivalence classes.
Consequently, the norm on $\Cod$ is given by
the quotient norm
\begin{equation*}
	\codnorm{\eqclass{u}{v}}
	= \sup_{\conorm{\varphi}\leq 1} \ddual{\varphi}{\eqclass{u}{v}} =
	\inf\set*{
		\mnorm{u + \mu} + \mdnorm{v + \nu}
		\given
		(\mu, \nu) \in C^\bot
	}
	.
\end{equation*} 
In the subsequent sections,
we frequently encounter elements of the form $\eqclass{u}{0}$ with $u \in \M = \C^*$.
Note that $\eqclass{u}{0}$ is the restriction of the functional $u \colon \C \to \R$
to the domain $\Co$.
In what follows, we give
a simpler representation of the dual norm for such elements $\eqclass{u}{0}$.
For this purpose, we denote by~$\Lip$ the space of Lipschitz continuous functions on~$\Omega$ which we equip with the canonical norm
\begin{equation*}
	\lnorm{\varphi} \coloneqq \max\set{\cnorm{\varphi},\lip(\varphi)},
\end{equation*}
where we used the Lipschitz constant $\lip(\varphi)$ of $\varphi$ defined in \eqref{eq:lipschitz_constant}.
Naturally, every~$u \in \M$ defines a linear function on~$\Lip$. The associated dual norm  
\begin{equation}
    \label{eq:BL}
    \blnorm{u}= \sup_{\lnorm{\varphi} \leq 1} \langle \varphi, u \rangle
\end{equation}
will be referred to as the \textit{bounded Lipschitz} norm.
Due to our standing assumption that $\Intr(\Omega)$
is uniformly locally quasiconvex,
the space $\Co$ is embedded into the Lipschitz functions, see \cref{lem:LUC_implies_C1_embeds_Lip},
and
we obtain the equivalence of the bounded Lipschitz norm with
the natural norm of $\Co$.
\begin{lemma}
	\label{lem:equivalence_of_norms_co}
	The norms
	$\codnorm{\eqclass{\cdot}{0}}$ and $\blnorm{\cdot}$
	are equivalent on $\M$.
\end{lemma}
\begin{proof}
	From
	\cref{lem:LUC_implies_C1_embeds_Lip}
	and the definition of $\lnorm{\cdot}$
	we have
	\begin{equation*}
		\lnorm{\varphi}
		\le
		L \conorm{\varphi}
		\qquad
		\forall \varphi \in \Co
	\end{equation*}
	for some constant $L > 0$.
	This yields
	\begin{equation*}
		\codnorm{\eqclass{u}{0}}
		=
		\sup_{\conorm{\varphi}\leq 1} \ddual{\varphi}{\eqclass{u}{0}}
		=
		\sup_{\conorm{\varphi}\leq 1} \dual{\varphi}{u}
		\le
		\sup_{\lnorm{\varphi}\leq L} \dual{\varphi}{u}
		=
		L \blnorm{u}.
	\end{equation*}
	For the reverse estimate,
	let $\varphi \in \Lip$ with $\lnorm{\varphi} \le 1$ be given.
	Using classical arguments,
	$\varphi$ can be extended to all of $\R^d$
	while preserving its Lipschitz constant.
	Now, let $\varepsilon > 0$ be given and denote by $\Phi_\varepsilon$
	a mollifier
	supported on $B_\varepsilon(0) \subset \R^d$.
	Then,
	$\psi := \Phi_\varepsilon \mathbin{\star} \varphi$
	satisfies
	$\psi \in \Co$,
	$\conorm{\psi} \le 1$
	and
	$\cnorm{\psi - \varphi} \le \varepsilon$.
	For an arbitrary $u \in \M$,
	we get
	\begin{equation*}
		\dual{\varphi}{u}
		\le
		\varepsilon \mnorm{u}
		+
		\dual{\psi}{u}
		=
		\varepsilon \mnorm{u}
		+
		\ddual{\varphi}{\eqclass{u}{0}}
		\le
		\varepsilon \mnorm{u}
		+
		\codnorm{u}.
	\end{equation*}
	Taking the limit $\varepsilon \searrow 0$
	and the supremum w.r.t.\ $\varphi$,
	this shows
	$\blnorm{u} \le \codnorm{u}$.
\end{proof}

Moreover, we require the notion of the Wasserstein-1 distance between positive measures of the same mass.
\begin{definition} \label{def:Wasserstein}
Given two measures~$\mu,\nu \in \M$ with~$\mu,\nu \geq 0$ and~$\mu(\Omega)=\nu(\Omega)$, we define the Wasserstein-1 distance between~$\mu$ and~$\nu$ as
\begin{equation*}
    W_1(\mu,\nu)= \min \set*{ \int_{\Omega \times \Omega} \abs{x-y}  \de \gamma(x,y) \given \gamma \in \Pi(\mu,\nu) },
\end{equation*}
where the minimum is taken over all couplings~$\gamma \in \Pi(\mu,\nu) \subset \MM(\Omega \times \Omega)$ between~$\mu$ and~$\nu$, i.e., there holds~$\gamma \geq 0$ as well as
\begin{equation*}
    \gamma(B \times \Omega )= \mu(B) \quad \text{and} \quad \gamma( \Omega \times B )= \nu(B) 
\end{equation*}
for all Borel sets~$B\in\BB(\Omega)$.
\end{definition}
The well known Kantorovich duality
yields
\begin{equation*}
	W_1(\mu,\nu) = \sup\set*{\int_\Omega \psi \de(\mu - \nu) \given \lip(\psi) \le 1}.
\end{equation*}

Using this, we can derive an optimal transport based representation of the bounded Lipschitz norm.
\begin{lemma} \label{lem:blnorm}
	For every~$u \in \M$ we denote by~$u=u_+-u_{-}$ the Jordan decomposition of~$u$.
	Then, we have
\begin{equation}
    \label{eq:nice_identity}
    \blnorm{u}= \min_{\substack{\mu, \nu \geq 0\\ \mu(\Omega)=\nu(\Omega)}} \left \lbrack W_1(\mu,\nu)+ \mnorm{u_+-\mu}+\mnorm{u_{-}-\nu} \right \rbrack.
\end{equation}
\end{lemma}
\begin{proof}
	The formula \eqref{eq:nice_identity} will be proven by Fenchel duality.
	We define the sets
	\begin{align*}
		X &\coloneqq \C^3, &
		D &\coloneqq \set{(\psi_1,\psi_2,\psi_3) \in X \given \lip(\psi_1)\le1,\; \cnorm{\psi_2} \le 1,\; \cnorm{\psi_3} \le 1},
		\\
		Y &\coloneqq \C^2, &
		M &\coloneqq \set{(\varphi_1, \varphi_2) \in Y \given \varphi_1 \le 0 ,\; \varphi_2 \le 0},
	\end{align*}
	as well as the mappings
	$A \colon X \to Y$, $f \colon X \to (-\infty,\infty]$ and $g \colon Y \to (-\infty,\infty]$ via
	\begin{align*}
		A(\psi_1,\psi_2,\psi_3) &\coloneqq (\psi_1 - \psi_2, -\psi_1 - \psi_3), &
		g(\varphi_1, \varphi_2) &\coloneqq I_{M}(\varphi_1, \varphi_2), \\
		f(\psi_1,\psi_2,\psi_3) &\coloneqq I_{D}(\psi_1, \psi_2, \psi_3) + \dual{\psi_2}{u_+} + \dual{\psi_3}{u_-}.
	\end{align*}
	The functionals $f$ and $g$ are convex and lower semicontinuous
	and $A$ is linear and bounded.
	The set $A \dom(f)$ contains the constant function $(-1,-1)$
	and this is a continuity point of $g$.
	Hence, Fenchel duality gives
	\begin{equation*}
		j_1 \coloneqq
		\inf_{(\psi_1,\psi_2,\psi_3) \in X} f(\psi_1, \psi_2, \psi_3) + g(A(\psi_1, \psi_2, \psi_3))
		=
		-\min_{(\mu,\nu) \in Y^*} f^*(-A^*(\mu,\nu)) + g^*(\mu,\nu)
		\eqqcolon j_2
		,
	\end{equation*}
	see \cite[Theorem~4.4.3]{BorweinZhu2005}.
	It remains to compute $j_1$ and $j_2$.
	For $j_1$ we get
	\begin{align*}
		j_1 &=
		\inf_{(\psi_1,\psi_2,\psi_3) \in X} f(\psi_1, \psi_2, \psi_3) + g(A(\psi_1, \psi_2, \psi_3))
		\\&=
		\inf\set*{
			\dual{\psi_2}{u_+} + \dual{\psi_3}{u_-}
			\given (\psi_1, \psi_2, \psi_3) \in D, \; -\psi_3 \le \psi_1 \le \psi_2
		}
		\\
		&=
		\inf\set*{
			\dual{\psi_1}{u_+} + \dual{-\psi_1}{u_-}
			\given (\psi_1, \psi_1, -\psi_1) \in D
		},
	\end{align*}
	where we used that $u_+, u_- \ge 0$.
	Now, $(\psi_1, \psi_1, -\psi_1) \in D$
	is equivalent to $\lnorm{\psi_1} \le 1$
	and, thus,
	\begin{equation*}
		j_1 = -\blnorm{u}.
	\end{equation*}

	To evaluate $j_2$, we first compute
	\begin{equation*}
		g^*(\mu, \nu)
		=
		\sup\set*{ \dual{\varphi_1}{\mu} + \dual{\varphi_2}{\nu} \given (\varphi_1, \varphi_2) \in M}
		=
		I_P(\mu,\nu),
	\end{equation*}
	where $P \subset \M^2$ is the subset of non-negative measures.
	Moreover,
	\begin{align*}
		f^*(-A^*(\mu,\nu))
		&=
		f^*(\nu - \mu, \mu, \nu)
		\\&
		=
		\sup\set*{ \dual{\psi_1}{\nu-\mu} + \dual{\psi_2}{\mu} + \dual{\psi_3}{\nu} - f(\psi_1, \psi_2, \psi_3) \given (\psi_1, \psi_2, \psi_3) \in X }
		\\&
		=
		\sup\set*{ \dual{\psi_1}{\nu-\mu} + \dual{\psi_2}{\mu-u_+} + \dual{\psi_3}{\nu-u_-} \given (\psi_1, \psi_2, \psi_3) \in D }
		\\&
		=
		\begin{cases}
			W_1(\mu,\nu) + \mnorm{u_+ - \mu} + \mnorm{u_- - \nu} & \text{if } \mu(\Omega) = \nu(\Omega) \\
			\infty & \text{if } \mu(\Omega) \ne \nu(\Omega)
		\end{cases}
	\end{align*}
	for all non-negative measures $(\mu, \nu) \in P$.
	Consequently,
	\begin{equation*}
		j_2
		=
		-\min\set*{ f^*(-A^*(\mu,\nu)) \given (\mu,\nu) \in P}
		=
		-\min_{\substack{\mu, \nu \geq 0\\ \mu(\Omega)=\nu(\Omega)}} \bracks*{ W_1(\mu,\nu)+ \mnorm{u_+-\mu}+\mnorm{u_{-}-\nu} }.
	\end{equation*}
	This finishes the proof of \eqref{eq:nice_identity}.
\end{proof}
A different proof for \eqref{eq:nice_identity}
can be obtained by using the ideas from
\cite[Proposition 24]{piccoli}.

\begin{example} \label{ex:blnorm}
Let~$u=\delta_{x_1}-\delta_{x_2}$ with~$x_1,x_2 \in \Omega$ and $x_1 \neq x_2$ be given. Then there holds~$\mnorm{u}=2$ but
\begin{equation*}
	\blnorm{u} =
	\min\set{2, \abs{x_1 - x_2}}
	,
\end{equation*}
see~\cite[Example 22]{piccoli}.
Similarly, we can check that for $x \in \Omega$, $h \in \R^d \setminus \set{0}$
with $x \pm h \in \Omega$
and $u = 2\delta_x - \delta_{x+h} - \delta_{x-h}$
we have
$\mnorm{u} = 4$ but
\begin{align*}
	\blnorm{u}
	=
	2 \min\set{2, \abs{h}}
	.
\end{align*}
\end{example}

\begin{remark}
	\label{rem:BL_and_W1}
	In the case that the diameter of $\Omega$
	is small enough
	and
	$u(\Omega) = 0$, we even have
	\begin{equation*}
		\blnorm{u} = W_1(u_+,u_-)
		.
	\end{equation*}
	Indeed, let $\psi \in \C$ with $\lip(\psi) \le 1$ be given
	such that
	$W_1(u_+,u_-) = \int_\Omega \psi \de u$.
	Note that the existence of $\psi$ follows from
	the Arzelà-Ascoli theorem.
	W.l.o.g., we can assume $\abs{\psi} \le \diam(\Omega) / 2$.
	Consequently,
	\begin{align*}
		W_1(\mu,\nu) - W_1(u_+,u_-)
		&\ge
		\int_\Omega \psi \de(\mu - \nu - u)
		=
		\int_\Omega \psi \de(\mu - u_+)
		-
		\int_\Omega \psi \de(\nu - u_-)
		\\&
		\ge
		-\frac{\diam(\Omega)}{2} \parens*{
			\mnorm{\mu - u_+}
			+
			\mnorm{\nu - u_-}
		}
		.
	\end{align*}
	This shows that the choice
	$\mu = u_+$ and $\nu = u_-$
	in \eqref{eq:nice_identity} is optimal
	whenever $\diam(\Omega) \le 2$.
\end{remark}

\subsection{A lifted setting} \label{subsec:liftedsetting}
In the following, let the canonical injection of~$\M$ into~$\Cod$ be denoted by
\begin{equation*}
\Pi \colon \M \to \Cod \quad \text{where} \quad \Pi u \coloneqq \eqclass{u}{0}.
\end{equation*}
Note that the continuity of this injection follows from the definitions.
Moreover, we define the lifted penalty term and forward operator by
\begin{equation} \label{def:extG}
G \colon \Cod \to [0,\infty] \quad \text{with} \quad G(\eqclass{u}{v})\coloneqq \begin{cases} \alpha \mnorm{\Tilde{u}} & \exists \tilde u \in \M \colon \eqclass{u}{v}=\Pi \tilde{u} \\  \infty &\text{else}  \end{cases}
\end{equation}
as well as 
\begin{equation}\label{def:extK}
    \Kb \colon \Cod \to Y \quad \text{with} \quad \Kb(\eqclass{u}{v})\coloneqq \int_\Omega k(x) \de u(x)+\int_\Omega Dk(x) \de v(x)  
\end{equation}
where the second integral has to be understood as
\begin{equation*}
    \int_\Omega Dk(x) \de v(x)\coloneqq\int_\Omega Dk(x) \frac{\de v}{\de \abs{v}}(x) \de \abs{v}(x).
\end{equation*}
The properties of $G$ and $\Kb$ will be investigated in
\cref{lem:propofliftG,lem:propofliftK}
below.

The range of $\Pi$ is not (weak*) closed.
Thus, it might not be possible to separate
elements $\eqclass{u}{v} \in \Cod \setminus \Ran(\Pi)$ from $\Ran(\Pi)$ by a continuous linear functional.
However, we can utilize the following characterization.
\begin{lemma}
	\label{lem:range_of_Pi}
	An element $\eqclass{u}{v} \in \Cod$ belongs to the range of $\Pi$
	if and only if is bounded by $\cnorm{\cdot}$, i.e., there exists~$C>0$ such that we have
	\begin{equation*}
		\abs{\ddual{\varphi}{\eqclass{u}{v}}} \le C \cnorm{\varphi} \qquad \forall \varphi \in \Co.
	\end{equation*}
\end{lemma}
\begin{proof}
	It is clear that ``$\Rightarrow$'' holds.
	In order to prove ``$\Leftarrow$'',
	let $\eqclass{u}{v}$ be bounded by $\cnorm{\cdot}$.
	Thus, this linear functional is a bounded functional
	on the dense subspace $\Co$ of the normed space $\C$.
	Hence, it can be uniquely extended to a functional $\tilde u \in \C^* = \M$,
	i.e., we have
	\begin{equation*}
		\ddual{\varphi}{\eqclass{u}{v}}
		=
		\dualM{\varphi}{\tilde u}
		=
		\ddual{\varphi}{\eqclass{\tilde u}{0}} \qquad \forall \varphi \in \Co. 
	\end{equation*}
	This proves $\eqclass{u}{v} = \Pi \tilde u$.
\end{proof}
Next, we show that $G$ is a nice functional.
\begin{lemma} \label{lem:propofliftG}
The function~$G$ from~\eqref{def:extG} is proper, convex and~weak* lower semi-continuous on~$\Cod$. Its preconjugate~$H \colon \Co \to [0, \infty] $ is given by
\begin{align*}
	\forall \varphi \in \Co:
	\qquad
	H(\varphi)
	=
	\begin{cases} 0 & \varphi \in \mathbb{B}_\alpha, \\ \infty & \text{else}, \end{cases}
\end{align*}
where~$\mathbb{B}_\alpha:=\set{ \phi \in \C \given \cnorm{\phi}\leq \alpha }$.
\end{lemma}
\begin{proof}
We first show~$H^*=G$. The claimed properties of~$G$ then follow immediately. Note that
\begin{equation*}
	H^*(\eqclass{u}{v}) = \sup\set*{ \ddual{\varphi}{\eqclass{u}{v}} \given \varphi \in \mathbb{B}_\alpha }.
\end{equation*}
We argue by a distinction of cases.
First, if $v = 0$, we immediately get~$H^*(\eqclass{u}{0}) = \alpha\mnorm{u}$
due to the density of $\mathbb{B}_\alpha \cap \Co$ in $\mathbb{B}_\alpha \subset \C$.
This shows that $H^*$ and $G$ coincide on the image of $\Pi$.
In case that $\eqclass{u}{v} \not \in \Im(\Pi)$,
we can utilize \cref{lem:range_of_Pi}
and obtain that for every $n \in \N$,
there exists $\varphi_n \in \Co$ with
\begin{equation*}
	\abs{\ddual{\varphi_n}{\eqclass{u}{v}}} \ge n \cnorm{\varphi_n}.
\end{equation*}
Since $\pm\alpha \varphi_n / \cnorm{\varphi_n} \in \mathbb{B}_\alpha$,
this implies $H^*(\eqclass{u}{v}) \ge n$.
Thus, we conclude $H^*(\eqclass{u}{v}) = G(\eqclass{u}{v}) = \infty$ if $\eqclass{u}{v} \not \in \Im(\Pi)$.
\end{proof}
Similarly, the operator $\Kb$ has nice properties
since it possesses a preadjoint.
\begin{lemma} \label{lem:propofliftK}
	The operator~$\Kb$ from~\eqref{def:extK} is well defined
	and sequentially weak*-to-strong continuous from~$\Cod$ to $Y$.
	Moreover, there holds~$\Kb=(\Kb_*)^*$ where~$\Kb_* \colon Y \to \Co$ satisfies~$\Kb_*y=K_* y$.
\end{lemma}
\begin{proof}
	The proof follows by the same argument as \cref{lem:proposofsparseop} noting that~$k \in C^{1,\gamma}(\Omega; Y)$. For the sake of brevity, further details are omitted.
\end{proof}
Now we consider the lifted minimization problem
\begin{align} \label{def:liftproblem}
    \tag{$\mathcal{P}$} 
    \min_{\eqclass{u}{v} \in \Cod} \mathcal{J}(\eqclass{u}{v}) \coloneqq \loss(\Kb(\eqclass{u}{v}))+G(\eqclass{u}{v}).
\end{align} 
The next proposition establishes its equivalence to Problem~\eqref{def:sparseintro}.
\begin{proposition} \label{prop:equivalence}
    An equivalence class~$\eqclass{u}{v}$ is a minimizer of Problem~\eqref{def:liftproblem} if and only if there hols~$\eqclass{u}{v}=\Pi \bar{u}$ for a solution~$\bar{u}$ of Problem~\eqref{def:sparseintro}.
\end{proposition}
\begin{proof}
    This follows immediately noting that~$\dom(\mathcal{J})= \Im(\Pi)$.
\end{proof}
\subsection{The main result} \label{subsec:mainresult}
We recall that we assume \cref{ass:functions} throughout the paper.
In the following, we assume the existence of a stationary point
\begin{align} \label{eq:sparsemeasSSC}
    \bar u= \sum^N_{j=1} \Bar{\lambda}_j \delta_{\Bar{x}_j} \quad \text{where} \quad \bar\lambda_j \neq 0, \bar{x}_j \in  \Intr(\Omega) \qquad\forall j=1,\dots,N
\end{align}
of \eqref{def:sparseintro}, i.e.,
the associated dual variable~$\bar{p}=-K_* \nabla \loss(K \bar{u})$ satisfies~$\Bar{p} \in \Co \cap \alpha \partial  \mnorm{\bar{u}} $.
Moreover, we define the sets~$\mathcal{A}$ as well as~$\I_+$ and~$\I_{-}$ by
\begin{equation}
    \label{eq:active_sets}
    \A \coloneqq \supp \Bar{u}=\{\Bar{x}_j\}^N_{j=1}, \quad \I_\pm \coloneqq \set*{x\in \Omega \given \Bar{p}(x)=\pm \alpha} \setminus \mathcal{A}. 
\end{equation}
Finally, set~$s_j =\sgn(\Bar{p}(\Bar{x}_j))$ for every~$\bar x_j \in \mathcal{A}$. The following theorem provides a tangible no-gap second-order condition for Problem~\eqref{def:sparseintro} at~$\Bar{u}$ as well as its equivalence to the quadratic growth of~$J$ w.r.t $\blnorm{\cdot}$ in the vicinity of~$\Bar{u}$.
We emphasize that we do not assume that the sets $\I_\pm$ are finite,
in particular, our setting is more general than the usual requirement \eqref{eq:strictcompintro},
which is equivalent to $\I_\pm = \emptyset$.
\begin{theorem} \label{thm:main}
Consider a measure~$\Bar{u}$ of the form~\eqref{eq:sparsemeasSSC} and assume that $\bar{p}=-K_* \nabla \loss(K \bar{u}) \in \Co \cap \alpha \partial  \mnorm{\bar{u}} $
is two times continuously differentiable around all~$\bar{x}_j \in \mathcal{A}$.
Then the following statements are equivalent.
\begin{enumerate}[label=\textnormal{(B\arabic*)}]
	\item \label{point:struc} There exists~$\theta>0$ with
		\begin{equation}
			\label{eq:def_hes_p}
			s_j \nabla^2 \Bar{p}(\Bar{x}_j) \leq_L -\theta \Id \qquad\forall \Bar{x}_j \in \A. 
		\end{equation}
		Moreover,
		\begin{equation}
			\label{eq:SOC2}
			\nabla^2 \loss(K \Bar{u}) \parens[\bigg]{\Kb \eqclass[\bigg]{\mu}{\sum_{j = 1}^n V_j \delta_{\bar x_j}}}^2
			-
			\sum_{j=1}^N  \bracks*{
				\frac{1}{\bar{\lambda}_j}
				\nabla^2\bar{p}(\bar{x}_j) {V}_j^2
			}
			>
			0
		\end{equation}
		holds for all
		$(\mu, V) \in \parens*{\mathcal C \times (\R^d)^N} \setminus \set{0}$,
		where the critical cone~$\mathcal{C} \subset \M$ is given by
		\begin{equation} \label{eq:condonlambda}
			\mathcal{C} \coloneqq \set*{\sum_{j = 1}^N \lambda_j \delta_{\Bar{x}_j}+\widehat{\mu}_+ -\widehat{\mu}_{-} \given \lambda \in \R^N,~ \widehat\mu_+ \in \mathcal{M}(\I_+)^+,~\widehat \mu_{-} \in \mathcal{M}(\I_-)^+  }.
		\end{equation}
	\item \label{point:quad} There exist~$\gamma>0$ and $\varepsilon >0$ such that there holds
		\begin{equation}
			\label{eq:quad_grwth}
			J(u)-J(\Bar{u}) \geq \gamma \blnorm{u-\bar{u}}^2 \qquad\forall u \in \M, \blnorm{u-\bar{u}} \leq \varepsilon. 
		\end{equation}
\end{enumerate}
\end{theorem}
The proof of this result is quite involved.
The implication ``\ref{point:quad}$\Rightarrow$\ref{point:struc}''
is given in \cref{lem:sscimplies}.
The reverse implication is given in \cref{sec:impliesSSC,sec:strucimplyNDC},
see also the comment after \cref{lem:sscimplies}.

In case that $\loss$ is convex,
the condition \eqref{eq:SOC2}
simplifies.
\begin{corollary} \label{coroll:nondegfromssc}
	In addition to the assumptions of \cref{thm:main},
	we assume that $\loss$ is convex.
	Consider the following statement.
	\begin{enumerate}[label=\textnormal{(C\arabic*)}]
		\item\label{item:convex_1}
			Condition \eqref{eq:def_hes_p} holds for some $\theta > 0$
			and
			\begin{equation} \label{eq:defofHess}
				(K\mu,\nabla^2 \loss(K \bar{u})K \mu )_Y >0  \qquad\forall \mu \in \mathcal{C} \setminus \set{0}
			\end{equation}
			is satisfied,
			where the critical cone $\mathcal C$ is defined in \eqref{eq:condonlambda}.
		\item\label{item:convex_2}
			There exist~$\gamma>0$ and $\varepsilon >0$ such that
			the quadratic growth condition \eqref{eq:quad_grwth} holds.
		\item\label{item:convex_3}
			Condition \eqref{eq:def_hes_p} holds for some $\theta > 0$
			and
			the matrix
			\begin{equation*}
				\parens*{
					( k(\bar x_i), \nabla^2 \loss(K \bar u) k(\bar x_j) )_Y
				}_{i,j = 1}^N
			\end{equation*}
			is positive definite.
		\item\label{item:convex_4}
			Condition \eqref{eq:def_hes_p} holds for some $\theta > 0$
			and
			$\operatorname{dim}\parens*{ \operatorname{span} \set*{k(\Bar{x}_1),\dots,k(\Bar{x}_N) } }=N$.
	\end{enumerate}
	Then,
	\ref{item:convex_1}
	and
	\ref{item:convex_2}
	are equivalent.
	In case that sets $\I_+$ and $\I_-$ are empty,
	i.e., if $\abs{\bar p(x)} < \alpha$ holds for all $x \in \Omega \setminus \mathcal{A}$,
	we get the equivalency
	of
	\ref{item:convex_1}--\ref{item:convex_3}.
	If, additionally, $\loss$ is strongly convex,
	we get the equivalency
	of
	\ref{item:convex_1}--\ref{item:convex_4}.
\end{corollary}
The equivalencies of
\ref{item:convex_1}, \ref{item:convex_3} and \ref{item:convex_4}
(under the respective assumptions)
are easy to check.
It thus remains to check the equivalency of
\ref{item:convex_1} and \ref{item:convex_2},
and this will be verified by showing that
\ref{item:convex_1}
is equivalent to
condition \ref{point:struc} from \cref{thm:main},
see \cref{lem:B1_implies_SSC} below.
Note that \ref{item:convex_4} corresponds to the
non-degeneracy of $\bar p$, see \eqref{eq:nondegintro}.

\begin{remark}
	\label{rem:convex_case}
	In order to prove \cref{thm:main},
	we use the lifted setting and work in the space $\Cod$.
	This is also apparent from the condition \eqref{eq:SOC2}.
	In the convex setting of \cref{coroll:nondegfromssc},
	however,
	this is no longer the case
	and its characterizations can be stated purely in the measure space $\M$,
	see \eqref{eq:defofHess}.
\end{remark}

In order to prove \cref{thm:main},
we are going to apply the theory from \cref{sec:nogapabstract}.
To this end, we first verify the Taylor formula \eqref{eq:hadamard_taylor_expansion}.
\begin{lemma}
	\label{lem:taylor_in_real_life}
	The functional $F(\eqclass{u}{v} ) \coloneqq \loss( \Kb( \eqclass{u}{v} ) )$ satisfies \eqref{eq:hadamard_taylor_expansion}
	with the choices
	\begin{align*}
		F'(\eqclass{\bar u}{0}) \eqclass{u}{v}
		&\coloneqq
		(\nabla \loss(K \bar u), \Kb \eqclass{u}{v} )_Y
		=
		\ddual{-\bar p}{\eqclass{u}{v} },
		\\
		F''(\eqclass{\bar u}{0}) \eqclass{u}{v}^2
		&\coloneqq
		\nabla^2 \loss( K \bar u) \parens*{ \Kb \eqclass{u}{v} }^2
	\end{align*}
	for arbitrary $\bar u \in \M$ and associated
	$\bar{p}=-K_* \nabla \loss(K \bar{u}) \in \Co$.
\end{lemma}
\begin{proof}
	This is clear as $\loss$ is assumed to be twice continuously Fréchet differentiable
	and $\Kb$ is linear and bounded.
\end{proof}
We first point out the following consequence of the abstract result in \cref{thm:abstractSSC}.
\begin{corollary} \label{coroll:nondegforsparse}
	Let a measure $\Bar{u} \in \M$ be given and set
	$\Bar{p} \coloneqq -K_* \nabla \loss(K \bar{u})$.
	Then the following statements are equivalent.
\begin{enumerate}
\item \label{point:SSC}  There holds
	the second-order condition
\begin{equation*} \label{eq:SSCmeasures}
(\Kb h,\nabla^2 \loss(K \Bar{u}) \Kb h)_Y   + G''( \eqclass{\bar{u}}{0}, \Bar{p}; h) >0 \qquad\forall h \in \Cod \setminus \{0\} \tag{SSC-$\mathcal{M}$}
\end{equation*}
and the non-degeneracy condition
\begin{equation*}
	\label{eq:NDCmeas}
	\tag{NDC-$\mathcal{M}$}
		\begin{aligned}
			&\text{for all $\seq{t_k} \subset (0,\infty)$, $\seq{h_k} \subset \Cod$
			with $t_k \searrow 0$, $h_k \weakstar 0$}
			\text{ and $\codnorm{h_k} = 1$, we have }\\
			&\qquad
			\liminf_{k \to \infty} \parens[\bigg]{
				\frac{1}{t^2_k} \parens*{G(\eqclass{\Bar{u}}{0}+t_k h_k)-G(\eqclass{\Bar{u}}{0})}
				-
				\frac{1}{t_k}\ddual{\Bar{p}}{ h_k}
			}
			> 0
			.
		\end{aligned}
\end{equation*}
\item There exist~$\gamma>0$ and $\varepsilon >0$ such that there holds
\begin{align*}
	J(u)-J(\Bar{u}) \geq \gamma \blnorm{u-\bar{u}}^2 \qquad\forall u \in \M, \blnorm{u-\bar{u}} \leq \varepsilon. 
\end{align*}
\end{enumerate}
\end{corollary}
\begin{proof}
	Due to \cref{lem:taylor_in_real_life},
	\cref{ass:abstract} is satisfied.
By definition, there holds
\begin{align*}
	F''(\eqclass{\bar u}{0}) h^2
	=
	(\Kb h,\nabla^2 \loss(K \Bar{u}) \Kb h)_Y
	=
	(\Kb h,\nabla^2 \loss(\Kb(\eqclass{(\Bar{u}}{0)})) \Kb h)_Y.
\end{align*}
Moreover, due to \cref{lem:propofliftK}, the mapping
\begin{align*}
    h \mapsto (\Kb h,\nabla^2 \loss(K \Bar{u}) \Kb h)_Y
\end{align*}
is sequentially weak*-continuous in~$\Cod$.
This shows that the second assumption of \cref{thm:abstractSSC}
is satisfied and
this also allows to drop the last term
\begin{equation*}
	\frac{1}{2}(\Kb h_k,\nabla^2 \loss(K\bar{u}) \Kb h_k)_Y
\end{equation*}
which would appear in the NDC.
Consequently, see \cref{thm:abstractSSC},~\eqref{eq:SSCmeasures} and~\eqref{eq:NDCmeas} are equivalent to
\begin{align*}
    \mathcal{J}(\eqclass{u}{v})-\mathcal{J}(\eqclass{\Bar{u}}{0}) \geq \tilde\gamma \codnorm{\eqclass{u-\Bar{u}}{v}}^2
		\qquad\forall \eqclass{u}{v}\in \Cod, \codnorm{\eqclass{u-\Bar{u}}{v}} \leq \varepsilon
\end{align*}
for some~$\tilde\gamma>0$ and~$\varepsilon>0$. The claimed statement now follows from~$\dom(\mathcal{J})= \Im(\Pi)$ as well as \cref{lem:equivalence_of_norms_co}.
\end{proof}
 
Thus, it suffices to show that \eqref{eq:SSCmeasures} together with~\eqref{eq:NDCmeas} are equivalent to~\ref{point:struc}.
Since this is rather technical, it will be split into several parts.
In this section, we show the implication of~\ref{point:struc} by~\eqref{eq:SSCmeasures} and~\eqref{eq:NDCmeas}.
\begin{lemma} \label{lem:sscimplies}
    Let~$\Bar{u}$ and~$\Bar{p}$ satisfy the assumptions of \cref{thm:main} and assume that~\eqref{eq:SSCmeasures} and~\eqref{eq:NDCmeas} hold.
		Then~\ref{point:struc} is satisfied.
\end{lemma}
\begin{proof}
We start by showing that~\eqref{eq:NDCmeas} implies the definiteness of the Hessian at all~$\Bar{x}_j \in \mathcal{A}$.
For this purpose,
we construct a suitable sequence~$\seq{h_k} \subset \Cod  $ with $\codnorm{h_k} =1$,~$k\in\N$,
as well as~$h_k \rightharpoonup^* 0 $ in~$\Cod$ and insert it into~\eqref{eq:NDCmeas}.
More in detail,
let~$\seq{t_k} \subset (0,\infty)$ denote an arbitrary sequence with~$t_k \searrow 0$
and fix~$\Bar{x}_j \in \A$ as well as~$v \in \R^d$,~$\abs{v}=1$.
Define
\begin{align*}
 \tilde h_k= \Pi \mu_k \quad \text{where} \quad \mu_k=\bar{\lambda}_j \frac{\left(\delta_{\Bar{x}_j-(t_k/\bar{\lambda}_j) v} - 2 \delta_{\bar{x}_j} + \delta_{\bar{x}_j+(t_k/\Bar{\lambda}_j) v}\right)}{2 t_k}
\end{align*}
where~$\bar{\lambda}_j \neq 0$ denotes coefficient associated to~$\delta_{\Bar{x}_j}$.
For $k$ large enough, we have~$(t_k/\bar{\lambda}_j) \abs{v} \leq 1$.
Thus, \cref{ex:blnorm} yields
$\blnorm{\mu_k}=1$. 
Now,
\cref{lem:equivalence_of_norms_co}
implies
$1 \le \codnorm{\tilde h_k}\le L$
for some $L > 0$, since~$\tilde h_k= \Pi \mu_k$.
Consequently,
$\sigma_k := \codnorm{\tilde h_k}^{-1}$
satisfies
$\sigma_k \in [L^{-1}, 1]$
and we set
\begin{equation*}
	h_k := \sigma_k \tilde h_k
	.
\end{equation*}

For arbitrary $\varphi \in \Co$,
we can use a Taylor expansion at $\bar x_j$ to obtain
\begin{align*}
	\lim_{k \rightarrow \infty} \ddual{\varphi}{\tilde h_k}
	=
	\lim_{k \rightarrow \infty} \Bar{\lambda}_j \frac{\left(\varphi(\Bar{x}_j-(t_k/\bar{\lambda}_j) v)-2 \varphi(\Bar{x}_j)+\varphi(\Bar{x}_j+(t_k/\bar{\lambda}_j) v)\right)}{2t_k}
	=
	0.
\end{align*}
Consequently,~$\tilde h_k \rightharpoonup^* 0$ and $h_k \rightharpoonup^* 0$ in~$\Cod$.

Finally, observe that there holds
\begin{align*}
     G(\eqclass{\Bar{u}}{0}+t_k h_k)-G(\eqclass{\Bar{u}}{0})=0
\end{align*}
for all~$k \in \N$ large enough.
Since $\bar p$ is assumed to be twice differentiable at $\bar x_j$,
we can perform a second-order Taylor expansion
to get
\begin{align*}
     \lim_{k \rightarrow \infty} - \frac{1}{t_k}\ddual{\Bar{p}}{\tilde h_k}  &=  \lim_{k \rightarrow \infty} -\Bar{\lambda}_j\frac{\left(\Bar{p}(\Bar{x}_j-(t_k/\bar{\lambda}_j) v)-2 \Bar{p}(\Bar{x}_j)+\Bar{p}(\Bar{x}_j+(t_k/\bar{\lambda}_j) v) \right)}{2t^2_k} \\ &
     = -\frac{\nabla^2 \Bar{p}(\Bar{x}_j)v^2}{2 \Bar{\lambda}_j}=-\sgn(\Bar{p}(\Bar{x}_j))\frac{\nabla^2 \Bar{p}(\Bar{x}_j)v^2}{2 \abs{\Bar{\lambda}_j}}
     .
\end{align*}
Summarizing the previous observations,
\eqref{eq:NDCmeas} implies 
\begin{equation*}
	-\sgn(\Bar{p}(\Bar{x}_j))\frac{\nabla^2 \Bar{p}(\Bar{x}_j)v^2}{2 \abs{\Bar{\lambda}_j}}
	=
	\liminf_{k \rightarrow \infty} \sigma_k^{-1} \parens*{
		\frac{1}{t^2_k} \left(G(\eqclass{\Bar{u}}{0}+t_k h_k)-G(\eqclass{\Bar{u}}{0})\right)- \frac{1}{t_k}\ddual{\Bar{p}}{ h_k}
	}
	>0.
\end{equation*}
Since ~$v\in\R^d$ with $\abs{v}=1$ was arbitrary,
this shows that the Hessian is definite at~$\Bar{x}_j \in \A$.
 
 Next, we prove that~\eqref{eq:SSCmeasures} implies the definiteness of~$\nabla^2 \loss(K \Bar{u})$ in the sense of~\eqref{eq:defofHess}.
 We do not argue by contradiction.
 Let a measure
 \begin{align*}
     \mu= \sum_{j\in \mathcal{A}} \lambda_j \delta_{x_j} + \widehat{\mu}_{+}-\widehat{\mu}_{-} \in \mathcal{C} \setminus \set{0}
 \end{align*}
 be given and denote by~$h= \Pi \mu$ the corresponding equivalence class.
 Then there holds~$h \neq 0$ and
 \begin{align*}
     0 \leq G''( \eqclass{\bar{u}}{0}, \Bar{p}; h)
     \leq \liminf_{k \to \infty}  \frac{G(\eqclass{\Bar{u}}{0}+t_k h)-G(\eqclass{\Bar{u}}{0})-t_k\ddual{\Bar{p}}{h}}{t^2_k/2}
 \end{align*}
 for an arbitrary but fixed sequence~$\seq{t_k} \subset (0,\infty)$ with~$t_k \searrow 0$. Now, for~$k$ large enough, we conclude
 \begin{align*}
     G(\eqclass{\Bar{u}}{0}&+t_k h)-G(\eqclass{\Bar{u}}{0})-t_k \ddual{\Bar{p}}{h}
     \\ &= \sum_{j\in \mathcal{A} } \left \lbrack \alpha \abs{\Bar{\lambda}_j+t_k \lambda_j}-\alpha\abs{\Bar{\lambda}_j}- t_k \lambda_j \Bar{p}(\Bar{x}_j) \right \rbrack+ \alpha t_k \mnorm{\widehat{\mu}_+} + \alpha t_k \mnorm{\widehat{\mu}_{-}}- t_k \ddual{\Bar{p}}{\widehat{\mu}_{+}-\widehat{\mu}_{-}}
		 \\
		 &= 0
 \end{align*}
 noting that~$\bar{p}(\Bar{x}_j)=s_j \alpha $, $\Bar{\lambda}_j \neq 0$ and $\abs{\Bar{\lambda}_j}=s_j \Bar{\lambda}_j$ for all~$\Bar{x}_j \in \mathcal{A}$ as well as
 \begin{align*}
     \langle \Bar{p},\widehat{\mu}_+ \rangle= \alpha \mnorm{\widehat{\mu}_+}, \quad \langle \Bar{p},\widehat{\mu}_- \rangle=- \alpha \mnorm{\widehat{\mu}_-}
 \end{align*}
 by construction of~$\I_{\pm}$.
 Consequently, \eqref{eq:SSCmeasures} implies
 \begin{align*}
     (\Kb h,\nabla^2 \loss(K \Bar{u}) \Kb h)_Y
     =
     (\Kb h,\nabla^2 \loss(K \Bar{u}) \Kb h)_Y   + G''( \eqclass{\bar{u}}{0}, \Bar{p}; h)
     >
     0
 \end{align*}
 for this particular choice of~$h$, finishing the proof.
\end{proof}
It remains to prove~$\ref{point:struc} \Rightarrow \eqref{eq:SSCmeasures}+\eqref{eq:NDCmeas}$.
For the sake of readability this will be done in two steps.
In \cref{sec:impliesSSC},
we compute the weak* second subderivative of $G$
and this enables us to prove that
\ref{point:struc} implies \eqref{eq:SSCmeasures},
see \cref{lem:B1_implies_SSC}.
As a byproduct, we also show
the equivalency of
\ref{point:struc}
and
\ref{item:convex_1},
see \cref{lem:B1_implies_SSC}.
The implication of \eqref{eq:NDCmeas}
by \ref{point:struc}
is addressed in \cref{sec:strucimplyNDC}.

We close this section with a final observation.
Note that \cref{thm:main} assumes the structure \eqref{eq:sparsemeasSSC}
of $\bar u$.
The next lemma shows that \eqref{eq:sparsemeasSSC}
is ``almost'' necessary for 
\eqref{eq:NDCmeas}
and, consequently,
for the quadratic growth condition \ref{point:quad}.
In fact, whenever $\bar u$ is not a finite sum of Dirac measures,
these conditions
\eqref{eq:NDCmeas} and \ref{point:quad}
cannot be true.

\begin{lemma}
	\label{lem:ndc_implies_isolated_support_2}
	Let $\bar u \in \M$ and $\bar p \in \C \cap \alpha\partial \mnorm{\bar u}$ be given.
	If \eqref{eq:NDCmeas} holds,
	then
	all points in $\supp(\bar u)$ are isolated.
	In particular,
	\begin{equation*}
		\bar u = \sum_{j = 1}^N \bar\lambda_j \delta_{\bar x_j}
	\end{equation*}
	for some $N \in \N$, $\bar\lambda_j \in \R$ and $\bar x_j \in \Omega$.
\end{lemma}
\begin{proof}
	We prove the contraposition.
	We suppose that $x \in \supp(\bar u)$ is not isolated.
	W.l.o.g., we assume $x \in \supp(\bar u_+)$.
	The case $x \in \supp(\bar u_-)$ can be handled analogously.

	From $\bar p \in \alpha \partial\mnorm{\bar u}$,
	we get $\bar p(x) = \alpha$.
	Due to the continuity of $\bar p$,
	$\supp(\bar u_-)$ has a positive distance
	$\delta > 0$
	from $x$.

	Since $x$ is not an isolated point of $\supp(\bar u)$,
	there exists a sequence $\seq{x_k} \subset \supp(\bar u) \setminus \set{x}$
	with $x_k \to x$.
	We set $s_k \coloneqq \abs{x_k - x}$
	and, w.l.o.g., we assume $(x_k - x) / s_k \to v \in \R^d$
	and $s_k \le s_{k-1} / 2$.
	Next, we choose
	\begin{align*}
		r_k &\coloneqq s_k / k
		,
		&
		\lambda_k &\coloneqq s_k / s_{k-1} \in [0,1/2]
		,
		\\
		\mu_k &\coloneqq \bar u_{| B_{r_k}(x_k)}
		,
		&
		m_k &\coloneqq \mnorm{\mu_k} > 0
		,
		\\
		\hat h_k &\coloneqq
		((1-\lambda_k) \delta_x + \lambda_k \delta_{x_{k-1}} ) m_k - \mu_k
		,
		&
		t_k &\coloneqq \blnorm{\hat h_k}
		,
		&
		h_k &\coloneqq \hat h_k / t_k
		.
	\end{align*}
	We want to give estimates for $t_k$.
	For deriving a lower bound, we use the definition \eqref{eq:BL}
	with
	$\varphi(\cdot) \coloneqq \abs{\cdot - x_k}$.
	Note that $\lnorm{\varphi} = 1$.
	Together with
	$ \abs{x_{k-1} - x_k} \ge s_{k-1} - s_k \ge s_k $
	we arrive at
	\begin{equation*}
		t_k
		=
		\blnorm{\hat h_k}
		\ge
		\dual{\varphi}{\hat h_k}
		\ge
		\parens*{
			(1-\lambda_k) \abs{x - x_k}
			+
			\lambda_k \abs{x_{k-1} - x_k}
			-
			r_k
		}
		m_k
		\ge
		\parens*{1 - \frac{1}{k}} s_k m_k
		\ge
		\frac12 s_k m_k
	\end{equation*}
	for all $k \ge 2$.
	For the upper bound, we use
	\eqref{eq:nice_identity}
	with $\mu = (\hat h_k)_+$, $\nu = (\hat h_k)_-$.
	Together with \cref{def:Wasserstein}
	and the coupling
	$\gamma = (\mu \otimes \nu) / m_k$,
	we arrive at
	\begin{align*}
		t_k
		\le
		W_1(\mu, \nu)
		\le
		\int_{\Omega \times \Omega} \abs{y - z} \de \gamma(y,z)
		&\le
		\parens*{
			(1-\lambda_k) \parens*{ \abs{x - x_k} + r_k }
			+
			\lambda_k \parens*{ \abs{x_{k-1} - x_k} + r_k }
		}
		m_k
		.
	\end{align*}
	In particular, $t_k \searrow 0$.

	By construction of $\hat h_k$, we get
	\begin{equation*}
		G(\eqclass{\Bar{u}+t_k h_k}{0})
		=
		G(\eqclass{\Bar{u}+\hat h_k}{0})
		=
		G(\eqclass{\Bar{u}}{0})
	\end{equation*}
	and from $x, x_{k-1} \in \supp(\bar u_+)$,
	we get
	$\bar p(x) = \bar p(x_{k-1}) = \alpha$
	and, consequently,
	\begin{equation*}
		\ddual{\bar p}{ \eqclass{\hat h_k}{0} } = 0.
	\end{equation*}
	In order to check that
	\eqref{eq:NDCmeas}
	does not hold,
	it remains to check that
	$\eqclass{h_k}{0} \weakstar 0$ in $\Cod$.
	Let $\varphi \in \Co$ be arbitrary.
	We have
	\begin{align*}
		\ddual{\varphi}{\eqclass{h_k}{0}}
		&=
		\frac{
			\parens*{
				(1-\lambda_k) \varphi(x) + \lambda_k \varphi(x_{k-1})
			} m_k
			- \int_{B_{r_k}(x_k)} \varphi \de\mu_k
		}{t_k}
		.
	\end{align*}
	From the Lipschitz continuity of $\varphi$,
	see \cref{lem:LUC_implies_C1_embeds_Lip},
	we
	get the estimate
	\begin{equation*}
		\abs*{
			\varphi(x_k) m_k
			-
			\int_{B_{r_k}(x_k)} \varphi \de\mu_k
		}
		\le
		r_k m_k \lnorm{\varphi}
		.
	\end{equation*}
	Consequently,
	\begin{equation*}
		\abs{\ddual{\varphi}{\eqclass{h_k}{0}}}
		\le
		\abs*{
			\frac{
				\parens*{
					(1-\lambda_k) \varphi(x) + \lambda_k \varphi(x_{k-1})
				} m_k
				- \varphi(x_k) m_k
			}{t_k}
		}
		+
		\frac{r_k m_k}{t_k} \lnorm{\varphi}
		.
	\end{equation*}
	Due to the choice of $r_k$
	and the lower bound on $t_k$,
	the last term goes to zero
	and, for brevity, we will replace it by $o(1)$.
	Together with the lower bound for $t_k$, we get
	\begin{align*}
		\frac12
		\abs{\ddual{\varphi}{\eqclass{h_k}{0}}}
		&\le
		\abs*{\frac{
				(1-\lambda_k) \varphi(x) + \lambda_k \varphi(x_{k-1})
				- \varphi(x_k)
			}{
				s_k
		}}
		+ o(1)
		.
	\end{align*}
	Owing to \cref{lem:quasiconvex_taylor},
	we get
	\begin{equation*}
		\abs*{
			\varphi(x_k) - \varphi(x) - \nabla \varphi(x)^\top (x_k - x)
		}
		\le
		\eta_k \abs{x_k - x}
		=
		\eta_k s_k
	\end{equation*}
	for some sequence $\seq{\eta_k} \subset [0,\infty)$
	with $\eta_k \searrow 0$.
	Consequently,
	\begin{equation*}
		\frac12
		\abs{\ddual{\varphi}{\eqclass{h_k}{0}}}
		\le
		\abs*{\frac{
				\nabla\varphi(x)^\top\parens*{\lambda_k (x_{k-1} - x) - (x_k - x)}
			}{
				s_k
		}}
		+
		\frac{
			\lambda_k \eta_{k-1} s_{k-1}
			+
			\eta_k s_k
		}{
			s_k
		}
		+ o(1)
		.
	\end{equation*}
	Now, we insert $\lambda_k = s_k / s_{k-1}$
	and get
	\begin{equation*}
		\frac12
		\abs{\ddual{\varphi}{\eqclass{h_k}{0}}}
		\le
		\abs{\nabla\varphi(x)}
		\abs*{
			\frac{x_{k-1} - x}{s_{k-1}}
			-
			\frac{x_k - x}{s_k}
		}
		+
		\eta_{k-1}
		+
		\eta_k
		+ o(1)
		\to
		0
		.
	\end{equation*}
	This shows
	$\eqclass{h_k}{0} \weakstar 0$.
	Consequently,
	\eqref{eq:NDCmeas} is violated.
\end{proof}

\section{Structural assumptions imply SSC} \label{sec:impliesSSC}
In this section, we
prove that the structural assumption
$\ref{point:struc}$
implies
the second-order condition~\eqref{eq:SSCmeasures}.
To this end, we explicitly characterize the weak* second subderivative of~$G$ at~$\eqclass{\Bar{u}}{0}$ for~$\Bar{p}$,
\begin{align*}
    G''(\eqclass{\Bar{u}}{0},\Bar{p};h)
    =
    \inf \set*{
        \liminf_{k \rightarrow \infty} \frac{G(\eqclass{\Bar{u}}{0}+t_k h_k)-G(\eqclass{\Bar{u}}{0})-t_k \ddual{ \Bar{p}}{h} }{t^2_k/2}
        \given
        t_k \searrow 0,~h_k \weakstar h
    },
\end{align*}
in all directions~$h \in \Cod$. For this purpose, and to avoid working with the weak* topology on~$\Cod$, we argue similarly to~\cite{wachsmuth2} and start by showing that the preconjugate~$H$, see \cref{lem:propofliftG},
is twice strong-strong epi-subdifferentiable. A candidate for~$G''(\eqclass{\Bar{u}}{0},\Bar{p};h)$ is then found by applying \cref{lem:conjugatelowergen}.
\subsection{Twice strong-strong epi-differentiability of~$H$} \label{subsec:epiofH}
We recall that the functional $H$ on $\Co$
is the indicator function of $\set{\phi \in \Co \given \cnorm{\phi} \le \alpha}$,
see \cref{lem:propofliftG}.
We start by proving an auxiliary result
concerning the tangent cone of this set.
\begin{lemma} \label{lem:tangential}
	We set
  $\mathcal{S} \coloneqq \set{\phi \in \Co \given \cnorm{\phi} \le \alpha}$.
	For every $p \in \mathcal{S}$,
	the tangent cone of $\mathcal{S}$ at $p$
	(in the sense of convex analysis)
	is given by
	\begin{equation*}
		\TangentCone_{\mathcal{S}}(p)
		=
		\set*{
			z \in \Co
			\given
			z \le 0 \text{ on } \set{p = \alpha}
			,\;
			z \ge 0 \text{ on } \set{p = -\alpha}
		}.
	\end{equation*}
\end{lemma}
\begin{proof}
  The inclusion ``$\subset$'' is clear and we just have to check ``$\supset$''.
  Thus, let $z$ from the right-hand side be given.
  For~$t>0$ we define the optimal values
  \begin{equation*}
    \bar{d}_{\pm} \coloneqq \max\set*{\pm z(x) \given x \in \Omega, \; p(x)=\pm\alpha}
    ,\qquad
    \bar{d}_{\pm,t} \coloneqq \max\set*{ t^{-1}(\pm p(x)-\alpha) \pm z(x) \given x \in \Omega }.
  \end{equation*}
  By assumption, there holds~$\Bar{d}_\pm\leq 0$.
  We readily verify~$\Bar{d}_{\pm,t} \rightarrow \Bar{d}  $
  and thus~$d_{\pm,t} \coloneqq \max\{0,\Bar{d}_{\pm,t}\} \rightarrow 0$ as~$t \rightarrow 0$.
  By construction, we further have
  \begin{equation*}
    t^{-1}(\pm p(x)-\alpha) \pm z(x)
    \le
    \bar d_{\pm, t}
    \le
    d_{\pm, t}
    \qquad\forall x \in \Omega.
  \end{equation*}
  Rearranging, this implies
  \begin{equation}
    \label{eq:nice_sign}
    \mathopen{} % The next \pm is a unary operator!
    \pm( p(x) + t z(x) \mp t d_{\pm,t} )
    \leq \alpha \qquad\forall x \in \Omega.
  \end{equation}

  Due to the continuity of $p$,
  the sets $\set{p \le -\alpha/2} = \set{-p \ge \alpha/2}$ and $\set{p \ge \alpha/2}$
  have a positive distance.
  Thus, we can find $\varphi_{\pm} \in \Co$
  with disjoint support,
  $0 \le \varphi_{\pm} \le 1$
  and $\varphi_{\pm} = 1$ on $\set{\pm p \ge \alpha/2}$.

  Now, we consider
  \begin{equation}
    \label{eq:for_the_tangent_cone}
    \xi_t \coloneqq - d_{+,t} \varphi_+ + d_{-,t} \varphi_-
    .
  \end{equation}
  The convergence of~$\xi_t$ towards~$0$ in~$\Co$ follows immediately
  and it remains to check $p + t (z + \xi_t) \in \mathcal{S}$ for $t > 0$ small enough.
  For $x \in \set{p \ge \alpha/2}$,
  we have
  \begin{equation*}
    (p + t (z + \xi_t))(x)
    =
    p(x) + t z(x) - t d_{+,t}
    \le
    \alpha
  \end{equation*}
  due to \eqref{eq:nice_sign}.
  For $x \not\in \set{p \ge \alpha/2}$,
  we get
  \begin{equation*}
    (p + t (z + \xi_t))(x)
    \le
    \frac\alpha2 + t \cnorm{z + \xi_t}
    \le
    \alpha
  \end{equation*}
  for $t > 0$ small enough.
  Similarly, we can handle the lower bound and this finishes the proof.
\end{proof}
It is interesting to note that tangent cone in $\Co$
does not see the derivative of the direction $z$
on the sets $\set{p = \pm \alpha}$.

In order to prepare the computation of the strong second subderivative of $H$,
we give an auxiliary lemma.
\begin{lemma}
	\label{lem:ersatz_fuer_lemma_sieben_punkt_sechs}
	We assume that $\bar u \in \M$
	is of the form \eqref{eq:sparsemeasSSC}.
	Further,
	$\bar p \in \Co \cap \alpha \partial  \mnorm{\bar{u}}$
	is given such that \eqref{eq:def_hes_p}
	holds,
	where $\A$ is defined in \eqref{eq:active_sets}.

	Then,
	there exists $r_0 > 0$ such that the following holds.
  \begin{enumerate}[label=(\arabic*)]
		\item\label{lem:ersatz_fuer_lemma_sieben_punkt_sechs:1}
			The balls
			$B_{r_0}(\bar x_j)$, $j = 1,\ldots, N$, are pairwise disjoint and subsets of $\Omega$.
		\item\label{lem:ersatz_fuer_lemma_sieben_punkt_sechs:2}
			For all $j \in \set{1,\ldots,N}$
			we have
			\begin{equation}
				\label{eq:hess_p_psd_locally_around_bar_x}
				\frac\alpha2
				\le
				s_j \bar p(x)
				\le
				\alpha - \frac\theta4 \abs{x - \bar x_j}^2
				\qquad
				\forall x \in B_{r_0}(\bar x_j)
				.
			\end{equation}
	\end{enumerate}
\end{lemma}
\begin{proof}
	It is clear that
	\ref{lem:ersatz_fuer_lemma_sieben_punkt_sechs:1}
	holds for $r_0 > 0$ small enough.
	In order to verify
	\ref{lem:ersatz_fuer_lemma_sieben_punkt_sechs:2}
	we can combine a Taylor expansion
	\begin{equation*}
		\bar p(x)
		=
		\bar p(\bar x_j)
		+
		\nabla\bar p(\bar x_j)^\top (x - \bar x_j)
		+
		\frac12 \nabla^2 \bar p(\bar x_j) (x - \bar x_j)^2
		+
		o(\abs{x - \bar x_j}^2)
	\end{equation*}
	with $\bar p(\bar x_j) = s_j \alpha$,
	$\nabla \bar p(\bar x_j) = 0$ and \eqref{eq:def_hes_p}.
\end{proof}
% One of the authors would like to emphasize
% that we do not need that $\bar p$ is twice continuously Fréchet differentiable at $\bar x_j$
% in the proof of \cref{lem:ersatz_fuer_lemma_sieben_punkt_sechs}.

Now we are in position to give a precise characterization of~$H''(\Bar{p},\eqclass{\Bar{u}}{0};z)$.
\begin{proposition} \label{prop:equidiffofJ}
	Under the assumptions of \cref{lem:ersatz_fuer_lemma_sieben_punkt_sechs},
the functional~$H$ is strongly-strongly twice epi-differentiable at~$\bar{p}$ for~$\eqclass{\bar{u}}{0}$ with
\begin{equation}
    \label{eq:second_subderivative_H}
H''(\bar{p},\eqclass{\bar{u}}{0},z)
=
-\sum_{j = 1}\bar{\lambda}_j \nabla^2 \bar{p}(\bar{x}_j)^{-1}\nabla z(\bar{x}_j)^2
+
I_\mathcal{Z}(z)
\qquad\forall z \in \Co
\end{equation}
where
\begin{equation*}
    \mathcal{Z}= 
    \set*{
        z \in \Co
        \given
        z \le 0 \text{ on } \I_+
        ,\;
        z \ge 0 \text{ on } \I_-
        ,\;
        z(\bar x_j) = 0 \;\forall j = 1,\ldots,N
    }.
\end{equation*}
\end{proposition}
\begin{proof}
Note that~$\bar{p} \in \mathbb{B}_\alpha$ and thus~$H(\bar{p})=0$. For abbreviation, given~$t_k>0$ and~$z_k \in\Co$, set 
\begin{align*}
D(t_k,z_k) \coloneqq \frac{H(\bar p+t_k z_k)-t_k \ddual{z_k}{\eqclass{\bar{u}}{0}}}{t^2_k/2}=\frac{H(\bar p+t_k z_k)-t_k\sum^N_{j=1} \bar{\lambda}_j z_k(\bar{x}_j)}{t^2_k/2}.
\end{align*}
Thus,
\begin{equation*}
    H''(\bar{p},\eqclass{\Bar{u}}{0};z)
    =
    \inf \set*{
        \liminf_{k \rightarrow \infty} D(t_k,z_k)
        \given
        t_k \searrow 0,~z_k \rightarrow z
    }
    \qquad\forall z \in \Co.
\end{equation*}
The proof is divided into several steps.

\textit{Step 1}:
We start by showing that~$H''(\bar{p},\eqclass{\bar{u}}{0};z)=\infty$ if~$z \not \in \mathcal{Z}$.
For this purpose, we first note that
\begin{align*}
    \mathcal{Z} = 
    \TangentCone_{\mathcal{S}}(\bar p)
    \cap
    \set*{
        z \in \Co
        \given
        z(\bar x_j) = 0 \;\forall j = 1,\ldots,N
    },
\end{align*}
cf.\ \cref{lem:tangential}.
In case $z \not\in \TangentCone_{\mathcal{S}}(\bar p)$,
for all sequences $\seq{t_k} \subset (0,\infty)$
and $\seq{z_k} \subset \Co$ with $t_k \searrow 0$ and $z_k \to z$ in $\Co$,
we have $\bar p + t_k z_k \not\in \mathcal{S}$ for all $k$ large enough.
Consequently,
$H(\bar p + t_k z_k) = \infty$
for all $k$ large enough
and, therefore,
$H''(p,\eqclass{\Bar{u}}{0};z) = \infty$
for all $z \not \in \TangentCone_{\mathcal{S}}(\bar p)$.

Now, for~$z \in \TangentCone_{\mathcal{S}}(\bar p) \setminus \mathcal{Z}$
there exists an index~$\bar{\jmath} \in \set{1,\dots,N}$
with~$z(\bar{x}_{\bar{\jmath}}) \neq 0$. This implies
$\sgn(\bar p(\bar x_{\bar\jmath})) z(\bar x_{\bar\jmath}) < 0$
and, thus,
$\bar\lambda_{\bar\jmath} z(\bar x_{\bar\jmath}) < 0$.
Moreover,
$\bar\lambda_{j} z(\bar x_{j}) \le 0$ for all $j \in \set{1,\ldots,N}$.
For all sequences $\seq{t_k} \subset (0,\infty)$ and $\seq{z_k} \subset \Co$
with $t_k \searrow 0$ and $z_k \to z$ in $\Co$, we have
\begin{equation*}
    \liminf_{k \rightarrow \infty} D(t_k,z_k)
    \ge
    \liminf_{k \rightarrow \infty} \frac{- \sum^N_{j=1}\bar{\lambda}_j z_k(\bar{x}_j) }{t^2_k/2}
    \ge
    \liminf_{k \rightarrow \infty} \frac{- \bar{\lambda}_{\bar{\jmath}} z_k(\bar{x}_{\bar{\jmath}})  }{t^2_k/2}
    =
    \infty.
\end{equation*}
Thus,
$H''(\bar{p},\eqclass{\Bar{u}}{0};z) = \infty$.

Summarizing the previous observations, we conclude~$H''(\bar{p},\eqclass{\bar{u}}{0};z)=\infty$ if~$z \not \in \mathcal{Z}$.

\textit{Step 2}: Next, we will show that 
$H''(\bar{p},\eqclass{\bar{u}}{0};z)$
can be bounded from below
by the right-hand side of \eqref{eq:second_subderivative_H}
for all~$z \in \mathcal{Z}$.
Let sequences $\seq{t_k} \subset (0,\infty)$ and $\seq{z_k} \subset \Co$
with
$t_k \searrow 0$ and $z_k \to z$ in $\Co$ be given.
We have to show that $\liminf_{k \to \infty} D(t_k, z_k)$
can be bounded from below by the right-hand side of \eqref{eq:second_subderivative_H}.
It is clear that we only have to consider sequences with
\begin{equation*}
    \bar{p}+t_k z_k \in \mathbb{B}_\alpha \qquad\forall k \in \N
\end{equation*}
in the following.
We have
\begin{equation}
    \label{eq:random_liminf}
    \liminf_{k \rightarrow \infty}D(t_k,z_k)
    =
    \liminf_{k \rightarrow \infty}\frac{ - \sum^N_{j=1}\bar{\lambda}_j t_k z(\bar{x}_{j}) }{t^2_k/2}
    \geq
    2 \sum_{j = 1}^N \liminf_{k \rightarrow \infty}\frac{ - \bar{\lambda}_j z(\bar{x}_{j}) }{t_k}.
\end{equation}
For all~$j \in \set{1,\ldots,N}$ we define the perturbation
\begin{equation*}
	\xi_{j,k} =- t_k\nabla^2 \bar{p}(\bar{x}_j)^{-1} \nabla z_k(\bar{x}_j).
\end{equation*}
Note that $\abs{\xi_{j,k}} = O(t_k)$ as $k \to \infty$.
Since~$\bar{x}_j \in \Intr(\Omega)$ there holds~$\bar{x}_j + \xi_{j,k} \in \Omega $
for $k$ sufficiently large.
Using the feasibility $\bar p + t_k z_k \in \mathbb{B}_\alpha$, we get
\begin{align} \label{eq:perteq}
	\alpha \geq s_j ( \bar{p}(\bar{x}_j+\xi_{j,k})+ t_k z_k(\bar{x}+\xi_{j,k}))
\end{align}
for all sufficiently large~$k$.
By Taylor expansions of~$\bar{p}$ and~$z_k$, respectively, at~$\bar{x}_j$, we get
\begin{align*}
	\MoveEqLeft
	\bar{p}(\bar{x}_j+\xi_{j,k})+ t_k z_k(\bar{x}_j+\xi_{j,k}) \\
	&=
	\bar{p}(\bar{x}_j)
	+ t_k z_k(\bar{x}_j)+ \nabla \bar{p}(\bar{x}_j)^\top \xi_{j,k}
	+ t_k \nabla z_k(\bar x_j)^\top \xi_{j,k}
	+\frac{1}{2} \nabla^2 \bar{p}(\bar x_j)\xi_{j,k}^2
	+ o(t_k^2)
	\\ &=
	s_j \alpha
	+ t_k z_k(\bar{x}_j)
	-\frac{t_k^2}{2} \nabla^2 \bar{p}(\bar x_j)^{-1} \nabla z_k(\bar x_j)^2
	+ o(t_k^2)
	.
\end{align*}
Plugging this into~\eqref{eq:perteq}, we arrive at
\begin{equation*}
\frac{ -\bar{\lambda}_j z_k(\bar{x}_j)}{t_k}
\geq
-\frac12 \bar{\lambda}_j \nabla^2 \bar{p}(\bar{x}_j)^{-1} \nabla z_k(\bar{x}_j)^2
+ o(1).
\end{equation*}
By combining this estimate with \eqref{eq:random_liminf}
and using~$z_k \rightarrow z $ in~$\Co$ yields the desired result.

	\textit{Step 3}:
	Finally, given~$z \in \mathcal{Z}$
	as well as an arbitrary sequence~$\seq{t_k} \subset (0,\infty)$ with $t_k \searrow 0$,
	we construct a sequence~$\seq{z_k} \subset \Co$ with~$z_k \rightarrow z$
	and which achieves the right-hand side of~\eqref{eq:second_subderivative_H}.

	By using the $\xi_k$ from \eqref{eq:for_the_tangent_cone}
	(with $t = t_k$),
	we get $\bar p + t_k(z + \xi_k) \in \mathbb{B}_\alpha$.
	However, we have to modify this function in the neighborhoods of $\bar x_j$
	in order to make the value of $\bar\lambda_j z_k(\bar x_j)$,
	which appears in $D(t_k, z_k)$,
	as large as possible.

	Let $r_0 > 0$ from \cref{lem:ersatz_fuer_lemma_sieben_punkt_sechs}
	be given.
	We define $r_k \coloneqq \sqrt{4 t_k \cnorm{z}/\theta}$.
	In the sequel, $k$ is large enough, such that $r_k \le r_0$.
	Now, we consider an arbitrary $j \in \set{1,\ldots,N}$ with $\bar p(\bar x_j) = \alpha$,
	i.e., $s_j = 1$.
	Combining the definition of $r_k$ with \eqref{eq:hess_p_psd_locally_around_bar_x},
	we get
	\begin{equation}
		\label{eq:random_2134887}
		(\bar p + t_k z)(x)
		\le
		\alpha - \frac \theta4 r_k^2 + t_k \cnorm{z}
		=
		\alpha
		\qquad
		\forall x \in B_{r_0}(\bar x_j) \setminus B_{r_k}(\bar x_j)
		.
	\end{equation}
	For all $x \in B_{r_k}(\bar x_j)$
	we have the Taylor estimates
	\begin{align*}
		\bar p(x) &\le \bar p(\bar x_j) + \nabla \bar p(\bar x_j)^\top (x - \bar x_j) + \frac12 \nabla^2 \bar p(\bar x_j) (x - \bar x_j)^2
		+ \frac{\varepsilon_k}2 \abs{x - \bar x_j}^2
		, \\
		z(x) &\le z(\bar x_j) + \nabla z(\bar x_j)^\top (x - \bar x_j) + \varepsilon_k \abs{x - \bar x_j}
	\end{align*}
	with $\varepsilon_k \searrow 0$.
	Multiplying the second inequality by $t_k$ and adding the first inequality,
	we obtain
	\begin{equation*}
		(\bar p + t_k z)(x) - \alpha
		\le
		\frac12 \nabla^2 \bar p(\bar x_j) (x - \bar x_j)^2
		+ t_k \nabla z(\bar x_j)^\top (x - \bar x_j)
		+ \varepsilon_k t_k \abs{x - \bar x_j}
		+ \frac{\varepsilon_k}2 \abs{x - \bar x_j}^2.
	\end{equation*}
	Using Young's inequality and \eqref{eq:def_hes_p}, we get
	\begin{align*}
		(\bar p + t_k z)(x) - \alpha
		&\le
		\frac12 \nabla^2 \bar p(\bar x_j) (x - \bar x_j)^2
		+ t_k \nabla z(\bar x_j)^\top (x - \bar x_j)
		+ \frac{\varepsilon_k}{2} t_k^2
		+ \varepsilon_k \abs{x - \bar x_j}^2
		\\&
		\le
		\frac{1 - 2 \varepsilon_k/\theta} 2 \nabla^2 \bar p(\bar x_j) (x - \bar x_j)^2
		+ t_k \nabla z(\bar x_j)^\top (x - \bar x_j)
		+ \frac{\varepsilon_k}{2} t_k^2
	\end{align*}
	for all $x \in B_{r_k}(\bar x_j)$.
	For $k$ large enough, $1 - 2 \varepsilon_k/\theta \ge 0$.
	Since $\nabla^2 \bar p(\bar x_j)$ is negative definite,
	all vectors $a,b \in \R^d$ obey
	the inequality
	$2 a^\top b \le -\nabla^2 \bar p(\bar x_j) a^2 - \nabla^2 \bar p(\bar x_j)^{-1} b^2$.
	Thus,
	\begin{equation*}
		(\bar p + t_k z)(x) - \alpha
		\le
		- \frac{t_k^2}{2(1 - 2 \varepsilon_k/\theta)} \nabla^2 \bar p(\bar x_j)^{-1} \nabla z(\bar x_j)^2
		+ \frac{\varepsilon_k}{2} t_k^2
		=: - t_k \tilde \xi_{j,k}
	\end{equation*}
	for all $x \in B_{r_k}(\bar x_j)$.
	Note that $\tilde \xi_{j,k} \le 0$.
	Together with \eqref{eq:random_2134887} we infer
	\begin{equation*}
		\forall x \in B_{r_0}(\bar x_j)
		:\quad
		(\bar p + t_k (z + \tilde \xi_{j,k}))(x) \le \alpha.
	\end{equation*}
	For $k$ large enough, we also get a lower bound from \eqref{eq:hess_p_psd_locally_around_bar_x}
	and this results in
	\begin{equation*}
		\forall x \in B_{r_0}(\bar x_j)
		:\quad
		\abs{ (\bar p + t_k (z + \tilde \xi_{j,k}))(x) } \le \alpha.
	\end{equation*}
	For $j \in \set{1,\ldots,N}$
	with $s_j = -1$, we can argue similarly and obtain
	the same estimate,
	with an appropriately defined $\tilde \xi_{j,k}$.

	In order to glue these estimates together,
	let~$\seq{\varphi_j}_{j=1}^N \subset \Co$ denote Urysohn functions
	with
	\begin{align*}
		\varphi_j(x) &\in [0,1] \quad \forall x \in \Omega,
		&
		\varphi_j(x) &=1 \quad \forall x \in B_{r_0/2} (\bar{x}_j),
		&
		\varphi_j(x) &=0 \quad \forall x \in \Omega \setminus B_{r_0} (\bar{x}_j) 
	\end{align*}
	for all $j = 1,\ldots, N$.
	We set
	\begin{equation*}
		z_k
		\coloneqq
		z
		+
		\xi_k
		+
		\sum_{j = 1}^N
		\varphi_j ( \tilde \xi_{j,k} - \xi_k)
		.
	\end{equation*}
	Note that $\bar p + t_k (z + \xi_k) \in \mathbb{B}_\alpha$
	already implies that $\abs{\bar p + t_k z_k} \le \alpha$
	outside of the balls $B_{r_0}(\bar x_j)$.
	For $j \in \set{1,\ldots,N}$,
	we get
	on the set $B_{r_0}(\bar x_j)$
	\begin{equation*}
		\bar p + t_k z_k
		=
		(1-\varphi_j) (\bar p + t_k ( z + \xi_k))
		+
		\varphi_j (\bar p + t_k (z + \tilde \xi_{j,k}))
		\in [-\alpha, \alpha].
	\end{equation*}
	Consequently,
	$ \bar p + t_k z_k \in \mathbb{B}_\alpha$,
	i.e., $H(\bar p + t_k z_k) = 0$.
	Further, $z_k \to z$ in $\Co$ is clear since $\xi_k \to 0$ and $\tilde \xi_{j,k} \to 0$.
	Finally,
	\begin{equation*}
		D(t_k, z_k)
		=
		-\sum_{j = 1}^N
		\frac{\bar\lambda_j z_k(\bar x_j)}{t_k/2}
		=
		-\sum_{j = 1}^N
		\frac{\bar\lambda_j \tilde \xi_{j,k}}{t_k/2}
		\to
		-\sum_{j = 1}^N
		\bar{\lambda}_j \nabla^2 \bar{p}(\bar{x}_j)^{-1}\nabla z(\bar{x}_j)^2
		.
	\end{equation*}
	Hence, $\seq{z_k} \subset \Co$ is the desired recovery sequence.
\end{proof}

We remark that expressions similar to \eqref{eq:second_subderivative_H}
arise
when one studies second-order tangent sets (w.r.t.\ the norm of $\C$)
to the set of non-negative functions in $\C$,
see \cite[Example~4.152]{BonnansShapiro2000}.
This is quite surprising,
since we have worked in $\Co$ in order to arrive at \eqref{eq:second_subderivative_H}.

\subsection{Twice weak* epi-differentiability of~$G$} \label{subsec:diffofG}
Now, we are in position to characterize the weak* second subderivative of~$G$.
For this purpose, we exploit the derived  representation of~$H''(\bar{p},\eqclass{\bar{u}}{0};\cdot)$
as well as \cref{lem:conjugatelowergen}.
\begin{proposition} \label{prop:lowerboundonepiofG}
	Under the assumptions of \cref{lem:ersatz_fuer_lemma_sieben_punkt_sechs},
we have
\begin{align} \label{eq_lowerboundonGpp}
 G''(\eqclass{\bar{u}}{0},\bar{p}, \eqclass{\mu}{\nu}) \geq - \sum^N_{j=1}  \frac{1}{\bar{\lambda}_j} \nabla^2\bar{p}(\bar{x}_j) {V}_j^2
\end{align}
whenever there holds
\begin{align*}
  \eqclass{\mu}{\nu}= \eqclass*{\sum^N_{j=1} c_j \delta_{\bar x_j} + \mu_+ + \mu_-}{\sum^N_{j=1} V_j \delta_{ \bar x_j} }
  \quad \text{for some}
  \quad
  \begin{aligned}
    (V_j,c_j) &\in\left( \R^d \times \R \right),~j=1,\dots,N,
    \\
    \mu_+ &\in \MM(\I_+)^+, \mu_- \in \MM(\I_-)^-.
  \end{aligned}
\end{align*}
Otherwise~$G''(\eqclass{\bar{u}}{0},\bar{p}, \eqclass{\mu}{\nu})=\infty$.
\end{proposition}
\begin{proof}
We utilize \cref{prop:equidiffofJ,lem:conjugatelowergen}.
In order to compute the conjugate,
observe that~$H''(\bar{p},\eqclass{\bar{u}}{0};\cdot)$ can be rewritten as a composition,
\begin{align*}
\frac12 H''(\bar{p},\eqclass{\bar{u}}{0};z)=- \sum_{j=1}^N \bar{\lambda}_j \nabla^2 \bar{p}(\bar{x}_j)^{-1}\nabla z(\bar{x}_j)^2+I_{\mathcal{Z}}(z)= g(Az),
\end{align*}
where we define
\begin{align*}
A \colon \Co \to \parens*{\R^d \times \R}^N \times C(\I_+) \times C(\I_-),
	\quad Az= \parens*{ \nabla z(\bar{x}_1), z(\bar{x}_1), \dots, \nabla z(\bar{x}_N), z(\bar{x}_N), z_{|\I_+}, z_{|\I_-} }
\end{align*}
as well as~$g \colon \parens*{\R^d \times \R }^N \times C(\I_+) \times C(\I_-) \to (-\infty, \infty]$ by
\begin{align*}
	g(V_1,c_1,\dots,V_N,c_N, w_+, w_- )
	=
	\sum^N_{j=1} \bracks*{ -\frac{\bar{\lambda}_j}{2} \nabla^2 \bar{p}(\bar{x}_j)^{-1}V_j^2+I_{\{0\}}(c_j) }
	+
	I_{\le 0}(w_+)
	+
	I_{\ge 0}(w_-)
	.
\end{align*}
where
\begin{align*}
    I_{\le 0}(w) = \begin{cases}
        0 & w  \leq 0,\\  \infty & \text{else},
    \end{cases}
    \qquad
    I_{\ge 0}(w) = \begin{cases}
        0 & w  \geq 0,\\  \infty & \text{else},
    \end{cases}
    \qquad
    I_{{0}}(c) = \begin{cases}
        0 & c=0 \\  \infty, & \text{else}.
    \end{cases}
\end{align*}
We claim that there holds
\begin{align} \label{eq:fullspace}
         \dom g - \Ran A =\parens*{\R^d \times \R }^N \times C(\I_+) \times C(\I_-).
\end{align}
    Indeed, consider an arbitrary element
    \begin{align*}
        \Psi =(V_1,c_1,\dots,V_N,c_N, w_1, w_2 ) \in \parens*{\R^d \times \R }^N \times C(\I_+) \times C(\I_-) 
    \end{align*}
    and let~$r>0$ be small enough such that the sets~$\seq{O_j}^{N+2}_{j=1}$, where
    \begin{align*}
        O_j=B_r(\Bar{x}_j),~j=1,\dots,N,\quad O_{N+1}=\I_+, \quad O_{N+2}=\I_-,
    \end{align*}
    are pairwise disjoint. Furthermore, let~$\seq{\varphi_j}^{N+2}_{j=1}$ denote a family of Urysohn functions subordinate to $\seq{O_j}^{N+2}_{j=1}$. Define  
    \begin{align*}
        \Phi=(V_1,0,V_2,0,\dots,V_N,0, w_1-\cnorm{w_1}, w_2+\cnorm{w_2}  ) \in \dom g
    \end{align*}
    as well as
    \begin{align*}
        z=-\sum^N_{j=1} c_j \varphi_j- \cnorm{w_1}\varphi_{N+1}+ \cnorm{w_2}\varphi_{N+2}. 
    \end{align*}
    We readily verify~$\Psi=\Phi-A(z)$, hence~\eqref{eq:fullspace} holds.
In particular, this implies $0 \in \Intr( \dom g - \Ran A )$.
Thus, the conjugate of $g \circ A$
is given by the infimal postcomposition,
cf.\ \cite[Theorem~15.27]{BauschkeCombettes2011}%
\footnote{This reference only deals with Hilbert spaces, but the proof can be applied to Banach spaces as well.},
i.e.,
\begin{align*}
    \MoveEqLeft
    \parens*{ \frac12 H''(\bar{p},\eqclass{\bar{u}}{0},\cdot)}^*(\eqclass{\mu}{\nu})
    \\&
    =
    \inf \set*{
        g^*(\bar V_1,\bar c_1,\dots,\bar V_N,\bar c_N, \bar\mu_+, \bar\mu_- )
        \given
        A^*(\bar V_1,\bar c_1,\dots,\bar V_N, \bar c_N, \bar\mu_+, \bar\mu_- )=\eqclass{\mu}{\nu}
    },
\end{align*}
where we again adapt the convention of~$\inf \emptyset=\infty$. We readily verify
\begin{align} \label{eq:helpform}
	A^*(\bar V_1,\bar c_1,\dots,\bar V_N,\bar c_N, \bar\mu_+, \bar\mu_- )
	=
	\eqclass*{
		\sum^N_{j=1}\bar c_j \delta_{\bar x_j}
		+
		\bar\mu_+
		+
		\bar\mu_-
	}{
		\sum^N_{j=1}\bar V_j \delta_{ \bar x_j}
	}.
\end{align}
This immediately implies~$G''(\eqclass{\bar{u}}{0},\bar{p}, \eqclass{\mu}{\nu})=\infty$ if~$\eqclass{\mu}{\nu}$ is not of the form
\begin{align*}
  \eqclass{\mu}{\nu}
  =
  \eqclass*{
    \sum^N_{j=1}\bar c_j \delta_{\bar x_j}
    +
    \bar\mu_+ + \bar\mu_-
  }{
    \sum^N_{j=1}\bar V_j \delta_{ \bar x_j}
  }
\end{align*}
for some $(\bar V_j, \bar c_j) \in \R^d \times \R$
and
$\bar\mu_\pm \in \MM(\I_\pm)$.
Next, we observe that the function $g$
is given by a sum of independent terms.
Consequently,
we can compute the conjugate
by summing the conjugates of the summands.
This yields
\begin{equation*}
  g^*(\bar{V}_1,\bar{c}_1,\dots,\bar V_N, \bar c_N, \bar\mu_+, \bar\mu_- ) 
  =
  -\sum_{j = 1}^N \frac{1}{2 \bar\lambda_j} \nabla^2 \bar p(\bar x_j) \bar V_j^2
  +
  I_{\ge 0}(\bar\mu_+)
  +
  I_{\le 0}(\bar\mu_-)
  .
\end{equation*}
One can check that $\Ran A$ is dense,
consequently, the adjoint $A^*$ is injective.
The claimed statement follows.
\end{proof}
Now, in order to show the twice weak* epi-differentiability of~$G$,
we need to construct a suitable recovery sequence which achieves equality in~\eqref{eq_lowerboundonGpp}. 
We start with a lemma in which we consider the contribution at a single point.
\begin{lemma}
\label{lem:weaksforpert}
Let the assumptions of \cref{lem:ersatz_fuer_lemma_sieben_punkt_sechs} be satisfied.
Let~$(V, c) \in \R^d \times \R$ and $j \in \set{1,\ldots,N}$ be given. We define
\begin{align*}
  \eqclass{\mu}{\nu} \coloneqq \eqclass{c \delta_{\bar x_j}}{V \delta_{\bar x_j}}.
\end{align*}
For an arbitrary sequence~$\seq{t_k} \subset (0,\infty)$ with $t_k \searrow 0$, we define
\begin{align*}
    \mu_k
  \coloneqq
    c \delta_{\bar x_j}
    +
    (\bar{\lambda}_j + t_k c) \frac{ \delta_{\bar{x}_j+(t_k/(\bar{\lambda}_j + t_k c )) V }- \delta_{\bar{x}_j}}{t_k}
\end{align*}
for $k$ large enough such that $\bar \lambda_j + t_k c \ne 0$.
Then there holds
\begin{align*}
\eqclass{\mu_k}{0} \rightharpoonup^* \eqclass{\mu}{\nu} \quad \text{in} \quad \Cod
\end{align*}
as well as
\begin{align} \label{def:helpstrict}
  \lim_{k \rightarrow \infty}
  \frac{
    G\parens[\big]{\eqclass*{\bar{\lambda}_j \delta_{\bar{x}_j}+t_k\mu_k}{0}}
    -
    G\parens[\big]{\eqclass*{\bar{\lambda}_j \delta_{\bar{x}_j}}{0}}
    -
    t_k \ddual{\bar{p}}{\eqclass{\mu_k}{0}}
  }{t^2_k/2}
  =
  - \frac{1}{\bar{\lambda}_j} \nabla^2 \bar{p}(\bar{x}_j)V^2
  .
\end{align}
\end{lemma}
\begin{proof}
In the sequel, $k$ is large enough such that $t_k < \abs{\bar\lambda_j}/\abs{c}$
and $\bar x_j + \tilde t_k V \in \Omega$.
In particular, $\bar\lambda_j + t_k c$ has the same sign as $\bar\lambda_j$
and, consequently, is not zero.
For abbreviation, we set
\begin{align*}
\tilde{t}_k= t_k/(\bar{\lambda}_j + t_k c) \to 0.
\end{align*}
In order to show the weak* convergence,
let~$\varphi \in \Co$ be arbitrary. Then we have
\begin{align*}
\ddual{\varphi}{\eqclass{\mu_k}{0}}= c \varphi(\bar x_j)+ \left(\varphi(\bar{x}_j+\tilde{t}_k V )- \varphi(\bar{x}_j) \right)/ \tilde{t}_k.
\end{align*}
Due to $\varphi \in \Co$, we arrive at
\begin{align*}
\lim_{k \rightarrow \infty} \ddual{\varphi}{\eqclass{\mu_k}{0}}= c \varphi(\bar x_j)+ V^\top \nabla \varphi(\bar{x}_j)=\ddual{\varphi}{\eqclass{\mu}{\nu}}.                                                                                                                                    
\end{align*} 
Next, we show~\eqref{def:helpstrict} in case~$V=0$. We have
\begin{align*}
  \frac{
    G\parens[\big]{\eqclass*{\bar{\lambda}_j \delta_{\bar{x}_j}+t_k\mu_k}{0}}
    -
    G\parens[\big]{\eqclass*{\bar{\lambda}_j \delta_{\bar{x}_j}}{0}}
    -
    t_k \ddual{\bar{p}}{\eqclass{\mu_k}{0}}
  }{t^2_k/2}
  =  \frac{\alpha\abs{\bar{\lambda}_j+t_k c}-\alpha \abs{\bar{\lambda}_j}-t_k s_j \alpha c }{(t^2_k/2)} 
  = 0,
\end{align*}
where we used
\begin{align*}
    \abs{\bar{\lambda}_j+t_k c}=s_j(\bar{\lambda}_j+t_k c)=\abs{\Bar{\lambda}_j}+s_j t_k c \quad \text{as well as} \quad  \bar{p}(\bar{x}_j)=s_j \alpha.
\end{align*}
Finally,
if~$V \neq 0$, we have
\begin{align*}
    t_k \mu_k &=
    -\bar\lambda_j \delta_{\bar x_j} + (\bar\lambda_j + t_k c) \delta_{\bar x_j + \tilde t_k V}
    ,
    &
    \bar\lambda_j \delta_{\bar x_j} + t_k \mu_k &=
    (\bar\lambda_j + t_k c) \delta_{\bar x_j + \tilde t_k V}
    .
\end{align*}
Thus, we arrive at
\begin{align*}
  \MoveEqLeft
  \frac{
    G\parens[\big]{\eqclass*{\bar{\lambda}_j \delta_{\bar{x}_j}+t_k\mu_k}{0}}
    -
    G\parens[\big]{\eqclass*{\bar{\lambda}_j \delta_{\bar{x}_j}}{0}}
    -
    t_k \ddual{\bar{p}}{\eqclass{\mu_k}{0}}
  }{t^2_k/2}
  \\&
  =
  \frac{
    \alpha \abs{\bar\lambda_j + t_k c} - \alpha \abs{\bar\lambda_j}
    +
    \bar\lambda_j \bar p(\bar x_j)
    - (\bar\lambda_j + t_k c) \bar p(\bar x_j + \tilde t_k V)
  }{t^2_k/2}
  \\&
  =
  \frac{
    \parens{\bar\lambda_j + t_k c} \bar p(\bar x_j)
    - (\bar\lambda_j + t_k c) \bar p(\bar x_j + \tilde t_k V)
  }{t^2_k/2}
  =
  \frac{2}{\parens{\bar\lambda_j + t_k c}} \frac{ \bar p(\bar x_j) - \bar p(\bar x_j + \tilde t_k V)}{\tilde t_k^2}
  \\&
  \to
  -\frac1{\bar\lambda_j} \nabla^2 \bar p(\bar x_j) V^2,
\end{align*}
where we used a second-order Taylor expansion of $\bar p$
at $\bar x_j$
together with $\nabla \bar p(\bar x_j) = 0$
in the last step.
This finishes the proof.
\end{proof}
\begin{theorem} \label{thm:twiceepiofG}
	We assume that $\bar u \in \M$
	is of the form \eqref{eq:sparsemeasSSC}.
	Further,
	$\bar p \in \Co \cap \alpha \partial  \mnorm{\bar{u}}$
	is given such that \eqref{eq:def_hes_p}
	holds,
	where $\A$ is defined in \eqref{eq:active_sets}.

Then, the functional~$G$ is twice weakly* epi-differentiable at $\eqclass{\bar{u}}{0}$ for $\bar p$.
We have
\begin{align}
    \label{eq:subderivative}
 G''(\eqclass{\bar{u}}{0},\bar{p}, \eqclass{\mu}{\nu})
 =
 - \sum_{j = 1}^N  \frac{1}{\bar{\lambda}_j} \nabla^2\bar{p}(\bar{x}_j) {V}_j^2
 ,
\end{align}
whenever there holds
\begin{align} \label{eq:finalformofmu}
  \eqclass{\mu}{\nu}= \eqclass*{\sum^N_{j=1} c_j \delta_{\bar x_j} + \mu_+ + \mu_-}{\sum^N_{j=1} V_j \delta_{ \bar x_j} }
  \quad \text{for some}
  \quad
  \begin{aligned}
    (V_j,c_j) &\in\left( \R^d \times \R \right),~j=1,\dots,N,
    \\
    \mu_+ &\in \MM(\I_+)^+, \mu_- \in \MM(\I_-)^-.
  \end{aligned}
\end{align}
Otherwise~$G''(\eqclass{\bar{u}}{0},\bar{p}, \eqclass{\mu}{\nu})=\infty$.
\end{theorem}
\begin{proof}
    From \cref{prop:lowerboundonepiofG}
    we already know that
    the second subderivative is $\infty$ if $\eqclass{\mu}{\nu}$ is not as in \eqref{eq:finalformofmu}
    and that
    ``$\ge$'' in \eqref{eq:subderivative} holds.
    It remains to construct a recovery sequence
    for $\eqclass{\mu}{\nu}$ as in \eqref{eq:finalformofmu}.
    To this end, let a sequence~$\seq{t_k} \subset (0,\infty)$ with~$t_k \searrow 0$ be given.
    In view of \cref{lem:weaksforpert},
    we define
    \begin{equation*}
    \mu_k
  \coloneqq
  \sum_{j = 1}^N \mu_{j,k} + \mu_+ + \mu_-
  \coloneqq
    \sum_{j = 1}^N \bracks*{
      c_j \delta_{\bar x_j}
      +
      (\bar{\lambda}_j + t_k c_j) \frac{ \delta_{\bar{x}_j+(t_k/(\bar{\lambda}_j + t_k c_j )) V_j }- \delta_{\bar{x}_j}}{t_k}
    }
    +
    \mu_+ + \mu_-
    \end{equation*}
    for $k$ large enough.
    From \cref{lem:weaksforpert}
    we get $\eqclass{\mu_k}{0} \weakstar \eqclass{\mu}{\nu}$.
    Moreover, for $k$ large enough, we have
    \begin{align*}
        G\parens[\big]{\eqclass*{\bar u+t_k\mu_k}{0}}
        &=
        \sum_{j = 1}^N G\parens[\big]{\eqclass*{\bar \lambda_j \delta_{\bar x_j} + t_k \mu_{j,k}}{0}}
        +
        t_k G(\mu_+)
        +
        t_k G(\mu_-)
        ,
        \\
        G\parens[\big]{\eqclass*{\bar u}{0}}
        &=
        \sum_{j = 1}^N G\parens[\big]{\eqclass*{\bar \lambda_j \delta_{\bar x_j}}{0}}
        ,
        \\
        \ddual{\bar{p}}{\eqclass{\mu_k}{0}}
        &=
        \sum_{j = 1}^N
        \ddual{\bar{p}}{\eqclass{\mu_{j,k}}{0}}
        +
        \alpha \mu_+(\I_+)
        +
        \alpha \abs{\mu_-(\I_-)}
        .
    \end{align*}
    Note that
    $G(\mu_+) = \mu_+(\I_+)$
    and
    $G(\mu_-) = \abs{\mu_-(\I_-)}$.
    Consequently,
    summing over
    \eqref{def:helpstrict}
    yields
    \begin{equation*}
      \lim_{k \rightarrow \infty}
      \frac{
        G\parens[\big]{\eqclass*{\bar u+t_k\mu_k}{0}}
        -
        G\parens[\big]{\eqclass*{\bar u}{0}}
        -
        t_k \ddual{\bar{p}}{\eqclass{\mu_k}{0}}
      }{t^2_k/2}
      =
      - \sum_{j = 1}^N \frac{1}{\bar{\lambda}_j} \nabla^2 \bar{p}(\bar{x}_j)V^2
      .
    \end{equation*}
    This proves the claim.
\end{proof}

Now, we can finally prove that~\ref{point:struc} implies~\eqref{eq:SSCmeasures}.
\begin{lemma}
	\label{lem:B1_implies_SSC}
	Let the assumptions of \cref{thm:main} be satisfied.
	Then,
	\ref{point:struc} implies \eqref{eq:SSCmeasures}.
	In case that $\loss$ is convex, \ref{point:struc} is equivalent to \ref{item:convex_1}.
\end{lemma}
\begin{proof}
According to \cref{thm:twiceepiofG}, it suffices to show that
\begin{align} \label{eq:sschel2}
    (\Kb h,\nabla^2 \loss(K \Bar{u}) \Kb h)_Y   + G''( \eqclass{\bar{u}}{0}, \Bar{p}; h) >0 
\end{align}
for all~$h = \eqclass{\mu}{\nu} \neq 0$
as in \eqref{eq:finalformofmu}.
Together with the expression \eqref{eq:subderivative}
for the second subderivative,
this is precisely condition \eqref{eq:SOC2}.

It remains to consider the situation
with a convex $\loss$.
We have to prove the equivalency of \eqref{eq:SOC2}
and \eqref{eq:defofHess}.
It is clear that \eqref{eq:SOC2}
implies \eqref{eq:defofHess}.

Now, let \eqref{eq:defofHess} be satisfied.
For an~$h = \eqclass{\mu}{\sum_{j = 1}^N V_j \delta_{\bar x_j}} \ne 0$ with
$\mu \in \mathcal{C}$ and $V \in (\R^d)^N$,
we note that
\begin{align*}
	(\Kb h,\nabla^2 \loss(K \Bar{u}) \Kb h)_Y   + G''( \eqclass{\bar{u}}{0}, \Bar{p}; h) &\geq  (\Kb h,\nabla^2 \loss(K \Bar{u}) \Kb h)_Y - \sum_{j=1}^N  \left \lbrack \frac{1}{\bar{\lambda}_j} \nabla^2\bar{p}(\bar{x}_j) {V}_j^2 \right \rbrack \\ &
	\geq  (\Kb h,\nabla^2 \loss(K \Bar{u}) \Kb h)_Y +\theta \sum_{j=1}^N \abs{V_j}^2.
\end{align*}
Due to the convexity of $\loss$,
the first addend on the right-hand side is nonnegative.
Thus, if there exists~$j \in \set{1,\ldots,N}$ with~$V_j \neq 0$,
the right-hand side is positive.
Otherwise,
$V_j =0$ for all~$j \in \set{1,\ldots,N}$.
Hence, we can apply \eqref{eq:defofHess}
and, again,
obtain the positivity of the right-hand side.
This shows that \eqref{eq:SOC2}
holds.
\end{proof}
\section{Structural assumptions imply NDC} \label{sec:strucimplyNDC}
It remains to show that the structural assumption
$\ref{point:struc}$
implies the non-degeneracy condition~\eqref{eq:NDCmeas}.
We give two fundamentally different proofs for this fact, one based on the transport-based representation of the bounded Lipschitz norm, see \cref{lem:blnorm}, and another utilizing the abstract result of 
\cref{lem:for_NDC}.
\subsection{Proof based on \cref{lem:blnorm}}
In the following, let assumption~\ref{point:struc} hold but assume that~\eqref{eq:NDCmeas} does not. W.l.o.g., and recalling the definition of the lifted functional~$G$,~\eqref{def:extG}, the latter implies that
there exists sequences $\seq{t_k} \subset (0,\infty)$, $\seq{\mu_k} \subset \M$
with
\begin{equation}
		\label{eq:contNDC}
		\begin{aligned}
			&t_k \searrow 0, \; \eqclass{\mu_k}{0} \rightharpoonup^* 0, \;
			\codnorm{\eqclass{\mu_k}{0}}= 1,\; \text{and}\\
			&\qquad
			\lim_{k \rightarrow \infty}\frac{\alpha \mnorm{\Bar{u}+t_k \mu_k}-\alpha \mnorm{\Bar{u}}- t_k \langle \bar{p},  \mu_k \rangle }{t_k^2}
			=\lim_{k \rightarrow \infty}\frac{\alpha \mnorm{\Bar{u}+t_k \mu_k}-  \langle \bar{p}, \Bar{u}+t_k  \mu_k \rangle}{t_k^2} =0.
		\end{aligned}
	\end{equation}
We want to lead this to a contradiction. For this purpose, let~$r_0>0$ denote the radius from \cref{lem:ersatz_fuer_lemma_sieben_punkt_sechs}. The following lemma is a direct consequence of \cref{lem:ersatz_fuer_lemma_sieben_punkt_sechs}, the definition of~$\I_\pm$ as well as the continuity of~$\Bar{p}$.
\begin{lemma} \label{lem:sepnondeg}
Let Assumption~\ref{point:struc} hold and denote by~$r_0>0$ the radius from \cref{lem:ersatz_fuer_lemma_sieben_punkt_sechs}. Then there exists a compact set~$\I$ such that:
  \begin{enumerate}[label=(\arabic*)]
		\item\label{lem:sepnondeg:1}
			The sets
			$B_{r_0}(\bar x_j)$, $j = 1,\ldots, N$, and~$\I$ are pairwise disjoint.
		\item\label{lem:sepnondeg:2}
			There holds
			\begin{equation}
    (\I_+ \cup \I_-) \subset \I \qquad \text{as well as} \qquad |\Bar{p}(x)| \leq \alpha-\sigma \qquad \forall x \in \Omega \setminus \left( \I \cup  \bigcup^N_{j=1} B_{r_0} \right)
			\end{equation}
   for some~$\sigma>0$.
	\end{enumerate}
\end{lemma}
We define by
\begin{align*}
 \mu_{k,\I} \coloneqq \mu_k( \cdot \cap \I )  , \quad \mu_{k,j} \coloneqq \mu_k( \cdot \cap B_{r_0}(\Bar{x}_j) ),
\end{align*}
the restrictions of~$\mu_k$ to~$\I$ and~$B_{r_0}(\Bar{x}_j)$,~$j=1,\dots,N$, respectively, as well as
\begin{align*}
    \Tilde{\mu}_k \coloneqq \mu_k-\mu_{k,\I} -\sum^N_{j=1} \mu_{k,j} \qquad \text{and} \qquad \mu^\bot_{k,j} \coloneqq \mu_{k,j}- \mu_{k,j}(\{\Bar{x}_j\}) \delta_{\Bar{x}_j}.  
\end{align*}
Furthermore, we denote by~$\mu^{\bot,\pm}_{k,j}$ the positive and negative part of~$\mu^\bot_{k,j}$, respectively, i.e., there holds
\begin{align*}
    \mu^\bot_{k,j}=\mu^{\bot,+}_{k,j}-\mu^{\bot,-}_{k,j}, \quad \mnorm{\mu^\bot_{k,j}}=\mnorm{\mu^{\bot,+}_{k,j}}+\mnorm{\mu^{\bot,-}_{k,j}},
\end{align*}
and recall that~$s_j=\sgn(\Bar{p}(\bar{x}_j))$,~$j=1,\dots,N$. Finally, we abbreviate
\begin{align*}
	D^k_1= \sum_{j=1}^N \left \lbrack \alpha \abs{\Bar{\lambda}_j +t_k \mu_{k,j}(\{\Bar{x}_j\}) }- s_j \alpha (\Bar{\lambda}_j +t_k \mu_{k,j}(\{\Bar{x}_j\})) \right \rbrack, ~  D^k_2=  \sum_{j=1}^N \left \lbrack \alpha \mnorm{\mu^{\bot,-s_j}_{k,j}}+s_j\langle \Bar{p}, \mu^{\bot,-s_j}_{k,j} \rangle  \right  \rbrack 
\end{align*}
as well as
\begin{align*}
	D^k_3=  \alpha \mnorm{\mu_{k,\I}}- \langle \Bar{p}, \mu_{k,\I} \rangle  , ~ D^k_4= \alpha \mnorm{\Tilde{\mu}_k}- \langle \Bar{p}, \Tilde{\mu}_k \rangle, \quad D^k_5= \sum_{j=1}^N \left \lbrack \alpha \mnorm{\mu^{\bot,s_j}_{k,j}}-s_j\langle \Bar{p}, \mu^{\bot,s_j}_{k,j} \rangle  \right \rbrack.
\end{align*}
The following estimates are imminent.
\begin{lemma} \label{lem:lowerbound}
    Let Assumption~\ref{point:struc} hold and let~$\seq{\mu_k}$,~$\seq{t_k}$ satisfy~\eqref{eq:contNDC}. Then there holds
\begin{align*}
    t^{-2}_k (\alpha \mnorm{\Bar{u}+t_k \mu_k}-  \langle \bar{p}, \Bar{u}+t_k  \mu_k \rangle) = t^{-2}_k  D^k_1+ t^{-1}_k \sum^5_{i=2} D^k_i   
\end{align*} 
as well as
\begin{align*}
	D^k_1 &\geq -2 \alpha \sum_{j=1}^N \min\{ s_j (\Bar{\lambda}_j +t_k \mu_{k,j}(\{\Bar{x}_j\})),0\} \geq 0,
	&
	D^k_2 &\geq \alpha \sum_{j=1}^N  \mnorm{\mu^{\bot,-s_j}_{k,j}} \geq 0,
	\\
	D^k_4 &\geq \sigma \mnorm{\Tilde{\mu}_k} \geq 0,
	&
	D^k_5 &\geq \frac{\theta}{4} \sum_{j=1}^N \int_{\Omega} \abs{\Bar{x}_j-x}^2 \de \mu^{\bot,s_j}_{k,j} \geq 0.
\end{align*}
\end{lemma}
\begin{proof}
We derive the lower bound
\begin{align*}
    t^{-2}_k (\alpha \mnorm{\Bar{u}+t_k \mu_k}-  \langle \bar{p}, \Bar{u}+t_k  \mu_k \rangle) \geq t^{-2}_k  D^k_1+ t^{-1}_k \sum^5_{i=2} D^k_i   
\end{align*}
noting that there holds
\begin{align*}
	\mnorm{\mu_k}=\mnorm{\Tilde{\mu}_k}+ \sum_{j=1}^N \left \lbrack \abs{\mu_{k,j}(\{\bar{x}_j)\}}+ \mnorm{\mu^{\bot,s_j}_{k,j}}+ \mnorm{\mu^{\bot,-s_j}_{k,j}}  \right \rbrack+ \mnorm{\mu_{k,\I}}.
\end{align*}
The lower estimates on~$D^k_i$ follow immediately noting that
\begin{align*}
	D^k_2 =   \sum_{j=1}^N \left \lbrack  \alpha \mnorm{\mu^{\bot,-s_j}_{k,j}}+\langle \abs{\Bar{p}}, \mu^{\bot,-s_j}_{k,j} \rangle \right \rbrack \geq \alpha \sum_{j=1}^N  \mnorm{\mu^{\bot,-s_j}_{k,j}}, \quad D^k_4= \alpha \mnorm{\Tilde{\mu}_k}- \langle \Bar{p}, \Tilde{\mu}_k \rangle \geq \sigma \mnorm{\Tilde{\mu}_k}   
\end{align*}
as well as
\begin{align*}
	D^k_5 &= \sum_{j=1}^N \left \lbrack \alpha \mnorm{\mu^{\bot,s_j}_{k,j}}-s_j\langle \Bar{p}, \mu^{\bot,s_j}_{k,j} \rangle  \right \rbrack = \sum_{j=1}^N \left \lbrack \alpha \mnorm{\mu^{\bot,s_j}_{k,j}}-\langle \abs{\Bar{p}}, \mu^{\bot,s_j}_{k,j} \rangle  \right \rbrack  \geq \frac{\theta}{4} \sum_{j=1}^N \int_{\Omega} \abs{\Bar{x}_j-x}^2 \de \mu^{\bot,s_j}_{k,j}
\end{align*}
due to \cref{lem:ersatz_fuer_lemma_sieben_punkt_sechs}.
\end{proof}
Using this auxiliary lemma, we argue that, for increasing~$k$,~$\mu_k$ strongly approximates the surrogate sequence
\begin{align*}
	\widehat{\mu}_k := \sum_{j=1}^N \left\lbrack -s_j \mu^{\bot,s_j}_{k,j}(\Omega) \delta_{\Bar{x}_j} +  s_j \mu^{\bot,s_j}_{k,j} \right \rbrack
\end{align*}
in~$\M$.
\begin{proposition} \label{lem:convtozero}
    Let assumption~\ref{point:struc} hold and let~$\seq{\mu_k}$,~$\seq{t_k}$ satisfy~\eqref{eq:contNDC}. Then there holds
\begin{align} \label{eq:convoffix}
	\lim_{k\rightarrow \infty} \left \lbrack \mnorm{\Tilde{\mu}_k}+ \mnorm{{\mu}_{k,\I}}+  \sum_{j=1}^N \left\lbrack \min\{ s_j (\Bar{\lambda}_j +t_k \mu_{k,j}(\{\Bar{x}_j\})),0\}+ \mnorm{\mu^{\bot,-s_j}_{k,j}} \right \rbrack  \right \rbrack=0
\end{align}
as well as
\begin{align} \label{eq:convoftrans}
	\lim_{k \rightarrow \infty} \left\lbrack \sum_{j=1}^N  \abs{s_j \mu_{k,j}(\{\Bar{x}_j\})+ \mu^{\bot,s_j}_{k,j}(\Omega)} \right \rbrack=0.
\end{align}
In particular, we have
\begin{align*}
\lim_{k \rightarrow \infty} \left\lbrack \codnorm{\eqclass{\mu_k-\widehat{\mu}_k}{0}}+ |\codnorm{\eqclass{\widehat{\mu}_k}{0}}-1|   \right \rbrack =0
\end{align*}
and the sequence~$\seq[\big]{\sum_{j=1}^N\mnorm{\mu^{\bot,s_j}_{k,j}}}$ is unbounded.
\end{proposition}
\begin{proof}
    We immediately conclude 
    \begin{align*}
			\lim_{k\rightarrow \infty} \left \lbrack \mnorm{\Tilde{\mu}_k}+  \sum_{j=1}^N \left\lbrack \min\{ s_j (\Bar{\lambda}_j +t_k \mu_{k,j}(\{\Bar{x}_j\})),0\}+ \mnorm{\mu^{\bot,-s_j}_{k,j}} \right \rbrack  \right \rbrack=0
    \end{align*}
    from the lower bounds on~$D^k_1$,~$D^k_2$ and~$D^k_4$ in \cref{lem:lowerbound}, respectively. Moreover, due to~\eqref{eq:contNDC} as well as \cref{lem:lowerbound}, we have
    \begin{align*}
      \lim_{k \rightarrow \infty} D^k_3=  \lim_{k \rightarrow \infty} \left\lbrack \alpha \mnorm{\mu_{k,\I}}-\langle \Bar{p} , \mu_{k,\I} \rangle \right \rbrack=0 .
    \end{align*}
    Let~$\{\varphi_j\}^N_{j=1} \cup \{\varphi_\I \} \subset \Co $ denote a family of smooth Urysohn functions subordinate to~$B_{r_0}(\Bar{x}_j)$,~$j=1,\dots, N$, and~$\I$.
   Then we have
    \begin{align*}
       0=\lim_{k \rightarrow \infty } \ddual{\Bar{p}\varphi_\I}{\eqclass{\mu_k}{0}}
			 =\lim_{k \rightarrow \infty } \left\lbrack\langle \Bar{p}, \mu_{k,\I} \rangle+\langle \Bar{p} \varphi_\I, \Tilde{\mu}_k \rangle \right \rbrack
			 = \lim_{k \rightarrow \infty }\langle \Bar{p}, \mu_{k,\I} \rangle.
       \end{align*}
    and consequently~$\lim_{k \rightarrow \infty } \mnorm{\mu_{k,\I}}=0$. Similarly, observe that
    \begin{align*}
     0= \lim_{k \rightarrow \infty}    s_j\ddual{ \varphi_j}{\eqclass{\mu_k}{0}} = \lim_{k \rightarrow \infty} \left\lbrack s_j \mu_{k,j}(\{\Bar{x}_j\})+ \mu^{\bot,s_j}_{k,j}(\Omega)- \mu^{\bot,-s_j}_{k,j}(\Omega)+ s_j \langle \varphi_j  , \Tilde{\mu}_k \rangle \right \rbrack
    \end{align*}
    for all~$j=1,\dots,N$.
    Using~\eqref{eq:convoffix}, we conclude~\eqref{eq:convoftrans}.
        The statements
\begin{align*}
\lim_{k \rightarrow \infty} \left\lbrack \codnorm{\eqclass{\mu_k-\widehat{\mu}_k}{0}}+ |\codnorm{\eqclass{\widehat{\mu}_k}{0}}-1|   \right \rbrack =0
\end{align*}
    immediately follow from~\eqref{eq:convoffix} and~\eqref{eq:convoftrans}, respectively, $\codnorm{\eqclass{{\mu}_k}{0}}=1$, as well as noting that 
		\begin{align*}
			\mu_k- \widehat{\mu}_k
			=
			\Tilde{\mu}_k
			+ \mu_{k,\I}
			+ \sum^N_{j=1}
			\left \lbrack
			-s_j \mu^{\bot, -s_j}_{k,j}+ (\mu_{k,j}(\{\Bar{x}_j\})+s_j \mu^{\bot,s_j}_{k,j}(\Omega)) \delta_{\Bar{x}_j}
				\right \rbrack
			\end{align*}
    and thus
    \begin{align*}
			\lim_{k \rightarrow }   \mnorm{\mu_k- \widehat{\mu}_k} \leq \lim_{k \rightarrow \infty} \left\lbrack \mnorm{\Tilde{\mu}_k}+ \mnorm{\mu_{k,\I}}+ \sum^N_{j=1} \left \lbrack \mnorm{ \mu^{\bot, -s_j}_{k,j}}+ |\mu_{k,j}(\{\Bar{x}_j\})+s_j \mu^{\bot,s_j}_{k,j}(\Omega)|  \right \rbrack \right \rbrack =0
    \end{align*}
		Finally, assume that $\sum_{j=1}^N\mnorm{\mu^{\bot,s_j}_{k,j}}$ is bounded. In view of~\eqref{eq:convoffix} and~\eqref{eq:convoftrans}, we then conclude that~$\seq{\mu_k}$ is bounded in~$\Moc$. Thus we have~$\mu_k \rightharpoonup^* 0$ in~$\Moc$. By a compactness argument, this contradicts~$\codnorm{\eqclass{\mu_k}{0}}=1$.
\end{proof} 
Thus, w.l.o.g, we restrict ourselves to sequences~$\seq{\mu_k}$ with~$\sum_{j=1}^N\mnorm{\mu^{\bot,s_j}_{k,j}}>0$ in the following. The next corollary shows that the latter grows at most at a rate of~$t^{-1}_k$.
\begin{corollary} \label{coroll:estoftrans}
	Let Assumption~\ref{point:struc} hold and let~$\seq{\mu_k}$,~$\seq{t_k}$ satisfy~\eqref{eq:contNDC}. For all~$k \in \N$ large enough as well as all~$j \in \set{1,\ldots,N}$, there holds
\begin{align*}
    \frac{\abs{\bar{\lambda}_j}+1/2}{t_k} \geq -s_j \mu_{k,j}(\{\Bar{x}_j\}) \quad \text{as well as} \quad \mnorm{\mu^{\bot,s_j}_{k,j}} \leq 1/2+ t^{-1}_k(\abs{\bar{\lambda}_j}+1/2).
\end{align*}
In particular, this implies
\begin{align*}
	\sum_{j=1}^N \mnorm{\mu^{\bot,s_j}_{k,j}} \leq N/2+ t^{-1}_k\left(\sum_{j=1}^N\abs{\bar{\lambda}_j}+N/2 \right) .
\end{align*}
\end{corollary}
\begin{proof}
 According to~\eqref{eq:convoffix} and~\eqref{eq:convoftrans}, there holds
 \begin{align*}
\mnorm{\mu^{\bot,s_j}_{k,j}}=\mu^{\bot,s_j}_{k,j}(\Omega) &\leq \abs{s_j \mu_{k,j}(\{\Bar{x}_j\})+ \mu^{\bot,s_j}_{k,j}(\Omega) }-s_j \mu_{k,j}(\{\Bar{x}_j\}) \leq 1/2 -s_j \mu_{k,j}(\{\Bar{x}_j\}) 
 \end{align*}
 as well as
 \begin{align*}
      s_j (\Bar{\lambda}_j +t_k \mu_{k,j}(\{\Bar{x}_j\}))=\abs{\Bar{\lambda}_j}+s_j t_k \mu_{k,j}(\{\Bar{x}_j\}) \geq -1/2
 \end{align*}
 for all~$k\in\N$ large enough. Rearranging yields the desired results.
\end{proof}
Finally, we argue that the lower bound on~$D^k_5$ in \cref{lem:lowerbound} decreases at most at a linear rate.
\begin{corollary} \label{coroll:lowerboundonD4}
   Let Assumption~\ref{point:struc} hold and let~$\seq{\mu_k}$,~$\seq{t_k}$ satisfy~\eqref{eq:contNDC}. Moreover, let~$c_1>0$ be such that
   \begin{align*}
       \blnorm{\mu} \geq c_1 \codnorm{\eqclass{\mu}{0}} \quad \forall \mu \in \M.
   \end{align*}
   For all~$k \in \N$ large enough, there holds
    \begin{align*}
			\frac{c_1}{2}    \left(\sum_{j=1}^N \mnorm{\mu^{\bot,s_j}_{k,j}}\right)^{-1} \leq \sum_{j=1}^N \int_{\Omega} \abs{\Bar{x}_j-x}^2 \de \mu^{\bot,s_j}_{k,j}
    \end{align*}
\end{corollary}
\begin{proof}
Note that the positive and negative part of the surrogate~$\widehat{\mu}_k$ are given by
\begin{align*}
	\widehat{\mu}^{+}_k= \sum_{j=1}^N   \mu^{\bot,s_j}_{k,j}+\sum_{j=1}^N  \mu^{\bot,s_j}_{k,j} (\Omega) \delta_{\bar{x}_j}, \quad \widehat{\mu}^{-}_k= \sum_{j=1}^N   \mu^{\bot,s_j}_{k,j}+\sum_{j=1}^N  \mu^{\bot,s_j}_{k,j} (\Omega) \delta_{\bar{x}_j}.
\end{align*}
By construction, there holds~$\mnorm{\widehat{\mu}^{+}_k}=\mnorm{\widehat{\mu}^{-}_k}$ and the canonical product measure
\begin{align*}
	\gamma_k = \sum_{j=1}^N \mu^{\bot,s_j}_{k,j} \otimes \delta_{\Bar{x}_j}+ \sum_{j=1}^N \delta_{\Bar{x}_j} \otimes \mu^{\bot,s_j}_{k,j} 
\end{align*}
satisfies
    \begin{align}
    \gamma_k (B \times \Omega )= \widehat{\mu}^{+}_k(B) \quad \text{as well as} \quad \gamma_k ( \Omega \times B )= \widehat{\mu}^{-}_k (B) 
\end{align}
for all~$B\in\mathcal{B}(\Omega)$ as well as
\begin{align} \label{eq:normgammak}
	\norm{\gamma_k}_{\mathcal{M}(\Omega \times \Omega)}= \sum_{j=1}^N \mnorm{\mu^{\bot,s_j}_{k,j}}>0.
\end{align}
Hence~$\gamma_k$ is an admissible transport plan between $\widehat{\mu}^{+}_k$~and $\widehat{\mu}^{-}_k$. By Jensen's inequality as well as the definition of the bounded Lipschitz norm, we further conclude
\begin{align*}
	\sum_{j=1}^N \int_{\Omega} \abs{\Bar{x}_j-x}^2 \de \mu^{\bot,s_j}_{k,j}&= \int_{\Omega \times \Omega} \abs{y-x}^2 \de \gamma_k (x,y) \\ & \geq  \norm{\gamma_k}_{\mathcal{M}(\Omega \times \Omega)}^{-1} \left(\int_{\Omega \times \Omega} \abs{y-x} \de \gamma_k  \right)^2 \geq \norm{\gamma_k}_{\mathcal{M}(\Omega \times \Omega)}^{-1} \blnorm{\widehat{\mu}_k}^2.
\end{align*}
The claimed statement now follows by definition of~$c_1$ as well as~$\lim_{k\rightarrow \infty}\codnorm{\eqclass{\widehat{\mu}_k}{0}}=1$ and~\eqref{eq:normgammak}.
\end{proof}
Combining \cref{coroll:estoftrans,coroll:lowerboundonD4}, we are now prepared to prove~\eqref{eq:NDCmeas}.
\begin{lemma}
	\label{lem:B1_implies_NCD_1}
	Let the assumptions of \cref{thm:main} be satisfied.
	Then,
	\ref{point:struc} implies \eqref{eq:NDCmeas}.
\end{lemma}
\begin{proof}
    Assume that \eqref{eq:NDCmeas} does not hold, i.e., there are sequences~$\seq{\mu_k}$ and~$\seq{t_k}$ satisfying~\eqref{eq:contNDC}. By \cref{lem:lowerbound}, this also implies
    \begin{align*}
			\lim_{k \rightarrow 0} t^{-1}_k \sum_{j=1}^N \int_{\Omega} \abs{\Bar{x}_j-x}^2 \de \mu^{\bot,s_j}_{k,j}=0.
    \end{align*}
    However, following Corollaries~\eqref{coroll:estoftrans} and~\eqref{coroll:lowerboundonD4}, we also have
    \begin{align*}
			\liminf_{k \rightarrow \infty } t^{-1}_k \sum_{j=1}^N \int_{\Omega} \abs{\Bar{x}_j-x}^2 \de \mu^{\bot,s_j}_{k,j} \geq \liminf_{k \rightarrow \infty } c_1/2  \left( t_k N/2+ \sum_{j=1}^N\abs{\bar{\lambda}_j}+N/2 \right) >0
    \end{align*}
    yielding a contradiction.
\end{proof}

\subsection{Proof based on \cref{lem:for_NDC}}
\label{sec:easy_ndc}
We give a second proof
which is based on \cref{lem:for_NDC}.

\begin{lemma}
	\label{lem:B1_implies_NCD_2}
	Let the assumptions of \cref{thm:main} be satisfied.
	Then,
	\ref{point:struc} implies \eqref{eq:NDCmeas}.
\end{lemma}
\begin{proof}
    In order to apply \cref{lem:for_NDC},
    we verify that the structural assumptions
    imply \eqref{eq:descent_lemma}.
    Let $r_0 > 0$ from \cref{lem:ersatz_fuer_lemma_sieben_punkt_sechs} be given.
    Let $\seq{\varphi_j}_{j=1}^N \subset \Co$ denote Urysohn functions
    with
    \begin{align*}
        \varphi_j(x) &\in [0,1] \quad \forall x \in \Omega,
        &
        \varphi_j(x) &=1 \quad \forall x \in B_{r_0/2} (\bar{x}_j),
        &
        \varphi_j(x) &=0 \quad \forall x \in \Omega \setminus B_{r_0} (\bar{x}_j) 
    \end{align*}
    for all $j = 1,\ldots, N$.
    We further set $\varphi_{N+1} \coloneqq \parens*{1 - \sum_{j = 1}^N \varphi_j} \alpha^{-1} \bar p$
    and $s_{N + 1} \coloneqq 1$.

    We set $\eta \coloneqq \min\set{ \alpha/8, \theta r_0^2 / 16} > 0$.
    Let $p \in \Co$ with $\conorm{p - \bar p} \le \eta$ be arbitrary.
    We define
    \begin{equation*}
        \beta_{N+1} \coloneqq \cnorm{p - \bar p}
        ,
        \qquad
        \beta_j \coloneqq \sup\set{ s_j p(x) \given x \in B_{r_0 / 2}(\bar x_j) } - \alpha
        \qquad\forall j = 1,\ldots, N
        .
    \end{equation*}
    Let us check that
    $\hat p := p - \sum_{j = 1}^{N+1} s_j \beta_j \varphi_j \in \mathbb{B}_\alpha$.
    For $x \in B_{r_0/2}(\bar x_j)$, $j \in \set{1,\ldots,N}$, we have
    \begin{equation*}
        \abs{\hat p(x)}
        =
        \abs{ p(x) - s_j \beta_j }
        =
        s_j p(x) - \beta_j
        \le
        \alpha.
    \end{equation*}
    For $x \in B_{r_0}(\bar x_j) \setminus B_{r_0/2}(\bar x_j)$, $j \in \set{1,\ldots,N}$, we have
    \begin{align*}
        \abs{\hat p(x)}
        &=
        \abs{p(x) - s_j \beta_j \varphi_j(x) - \beta_{N+1} \varphi_{N+1}(x)}
        =
        s_j p(x) - \beta_j \varphi_j(x) - s_j \cnorm{p - \bar p} (1 - \varphi_j(x)) \alpha^{-1} \bar p(x)
        \\&
        \le
        s_j p(x)
        \le
        s_j \bar p(x) + \cnorm{p - \bar p}
        \le
        \alpha - \frac{\theta}{4} \parens*{\frac{r_0}{2}}^2 + \cnorm{p - \bar p}
        \le \alpha,
    \end{align*}
    where we used \eqref{eq:hess_p_psd_locally_around_bar_x}.
    Finally, for $x \in \Omega \setminus \bigcup_{j = 1}^N B_{r_0}(\bar x_j)$
    we have
    \begin{align*}
        \abs{\hat p(x)}
        &=
        \abs{p(x) - \beta_{N+1} \alpha^{-1} \bar p(x)}
        \le
        \cnorm{p - \bar p} + \abs{\bar p(x)} \abs{ 1 - \beta_{N+1} \alpha^{-1}}
        \\&
        \le
        \cnorm{p - \bar p} + \alpha \abs{ 1 - \cnorm{p - \bar p} \alpha^{-1}}
        =
        \alpha
        .
    \end{align*}
    This shows
    $\hat p = p - \sum_{j = 1}^{N+1} s_j \beta_j \varphi_j \in \mathbb{B}_\alpha = \dom(H)$.

    Now, let $\tilde H$ and $H_2$ be as in \cref{lem:for_NDC} with $K \coloneqq N + 1$,
    $\bar x \coloneqq \eqclass{\bar u}{0}$
    and
    $\zeta_i \coloneqq \varphi_i$.
    Thus,
    \begin{equation*}
        \tilde H(p)
        \le
        H\parens[\Bigg]{
            p - \sum_{j = 1}^{N+1} s_j \beta_j \varphi_j
        }
        +
        H_2\parens[\Bigg]{
            \sum_{j = 1}^{N+1} s_j \beta_j \varphi_j
        }
        =
        H_2\parens[\Bigg]{
            \sum_{j = 1}^{N+1} s_j \beta_j \varphi_j
        }
        =
        \frac12 \sum_{j = 1}^{N+1} \beta_j^2 + \sum_{j = 1}^{N} \beta_j \abs{\bar\lambda_j}
        ,
    \end{equation*}
    where we used
    $\ddual{\varphi_{j}}{\eqclass{\bar u}{0}} = \dual{\varphi_{j}}{\bar u} = \bar\lambda_j$,
    for all $j = 1,\ldots, N$
    and
    $\ddual{\varphi_{N+1}}{\eqclass{\bar u}{0}} = \dual{\varphi_{N+1}}{\bar u} = 0$.
    Moreover,
    \begin{equation*}
        \dual{p - \bar p}{\bar u}
        =
        \sum_{j = 1}^N (p(\bar x_j) - \alpha s_j) \bar\lambda_j
        =
        \sum_{j = 1}^N (s_j p(\bar x_j) - \alpha) \abs{\bar\lambda_j}
        .
    \end{equation*}
    Hence,
    \begin{align*}
        \tilde H(p)
        -
        H(\bar p)
        -
        \dual{p - \bar p}{\bar u}
        &\le
        \frac12 \sum_{j = 1}^{N+1} \beta_j^2 + \sum_{j = 1}^N \beta_j \abs{\bar\lambda_j}
        -
        \sum_{j = 1}^N (s_j p(\bar x_j) - \alpha) \abs{\bar\lambda_j}
        \\&
        =
        \frac12 \sum_{j = 1}^{N+1} \beta_j^2 +
        \sum_{j = 1}^N (\beta_j - s_j p(\bar x_j) + \alpha) \abs{\bar\lambda_j}
        .
    \end{align*}
    For the first addend, we can directly use
    $\abs{\beta_j} \le \cnorm{p - \bar p}$.
    For the second addend, we consider
    \begin{align*}
        \beta_j - s_j p(\bar x_j) + \alpha
        &=
        \sup\set{ s_j p(x) \given x \in B_{r_0}(\bar x_j) } - s_j p(\bar x_j)
        \\&=
        \sup\set{ s_j (p(x) - p(\bar x_j)) \given x \in B_{r_0}(\bar x_j) }
        .
    \end{align*}
    For any $x \in B_{r_0}(\bar x_j)$ we use \eqref{eq:hess_p_psd_locally_around_bar_x} to obtain
    \begin{align*}
        s_j \parens{
            p(x) - p(\bar x_j)
        }
        &=
        s_j \parens*{
            \int_0^1 \nabla p(\bar x_j + t (x - \bar x_j))^\top (x - \bar x_j) \de t
        }
        \\&
        \le
        \conorm{p - \bar p} \abs{x - \bar x_j}
        +
        s_j \parens*{
            \int_0^1 \nabla\bar p(\bar x_j + t (x - \bar x_j))^\top (x - \bar x_j) \de t
        }
        \\&
        =
        \conorm{p - \bar p} \abs{x - \bar x_j}
        +
        s_j ( \bar p(x) - \bar p(\bar x_j) )
        \\&
        \le
        \conorm{p - \bar p} \abs{x - \bar x_j}
        -
        \frac\theta4 \abs{x - \bar x_j}^2
        \le
        \frac{1}{\theta} \conorm{p - \bar p}^2
        .
    \end{align*}
    By collecting the inequalities from above,
    we obtain that \eqref{eq:descent_lemma}
    is satisfied
    for some $\Lambda > 0$.
    Thus, the proof is finalized by the invocation of \cref{lem:for_NDC}.
\end{proof}
\section*{Acknowledgements}
The authors would like to thank Hannes Meinlschmidt for pointing out the notion of ``uniform local quasiconvexity''.
\appendix
\section{Uniform local quasiconvexity} \label{app:quasi}
In this appendix,
we give the proofs of the technical lemmas from
\cref{subsec:quasiconvex}.

\begin{proof}[Proof of \cref{lem:Lipschitz_implies_ulq}]
	Since $\partial D$ is compact,
	we find $n \in \N$,
	$\seq{p_i}_{i = 1}^n \subset \partial D$
	and
	associated radii
	$\seq{r_i}_{i = 1}^n \subset (0,\infty)$
	(as in the assumption)
	such that
	\(
		\partial D
		\subset
		\bigcup_{i = 1}^n U_{r_i/2}(p_i)
	\).
	The set $D \setminus \bigcup_{i = 1}^n U_{r_i/2}(p_i)$
	is compact and does not intersect $\partial D$.
	Thus
	\begin{equation*}
		r_0
		:=
		\inf\set*{
			\abs{x - p}
			\given
			x \in D \setminus \bigcup_{i = 1}^n U_{r_i/2}(p_i),
			\;
			p \in \partial D
		}
	\end{equation*}
	is positive.
	We set $r := \min\set{r_0, r_1/2, \ldots, r_n/2}$.
	Moreover,
	let $L > 0$ be a common Lipschitz constant
	for $l_{p_i}$ and $l_{p_i}^{-1}$, $i \in \set{1,\ldots,n}$.
	We define $C := L^2$.

	We check that $D$ is uniformly locally quasiconvex
	with parameters $r$ and $C$.
	To this end, let $x,y \in D$
	with $\abs{x - y} \le r$ be given.

	\underline{Case 1:} $x \in B_{r_i/2}(p_i)$ for some $i \in \set{1,\ldots,n}$.
	Due to $\abs{x - y} \le r \le r_i/2$,
	we have $x,y \in B_{r_i}(p_i)$.
	Consequently,
	$l_{p_i}(x), l_{p_i}(y) \in B_1^+(0)$
	can be joined by a line segment
	in $B_1^+(0)$
	of length
	$\abs{l_{p_i}(x) - l_{p_i}(y)} \le L \abs{x - y}$.
	Under the map $l_{p_i}^{-1}$,
	this becomes a Lipschitz curve in $D$ of length at most $L^2 \abs{x - y}$
	joining $x$ and $y$, as desired.

	\underline{Case 2:} $x \not\in B_{r_i/2}(p_i)$ for all $i \in \set{1,\ldots,n}$.
	In this case,
	$x \in D \setminus \bigcup_{i=1}^n U_{r_i/2}(p_i)$
	and
	$y \in U_r(x)$.
	As $U_r(x)$ does not intersect $\partial D$,
	we can simply join
	$x$ and $y$ by a line segment of length $\abs{x - y}$.
\end{proof}

Now, we discuss the consequences of $\Intr(\Omega)$
being uniformly locally quasiconvex.
Let $\gamma \in C([0,1]; \Intr(\Omega))$
be a Lipschitz continuous curve
as in \cref{def:LUQ}.
Moreover,
let $\varphi \in C^1(\Omega)$
be given.
Then, $\varphi \circ \gamma$
is absolutely continuous
and
its derivative at $t$ is given by
$\nabla \varphi(\gamma(t))^\top \gamma'(t)$
for almost all $t \in [0,1]$.
Consequently,
\begin{equation}
	\label{eq:HDI}
	\varphi(\gamma(1))
	-
	\varphi(\gamma(0))
	=
	\int_0^1 \nabla \varphi(\gamma(t))^\top \gamma'(t) \de t
	.
\end{equation}
This identity will be crucial in the next two proofs.

The next proof is similar to
a part of the proof of
\cite[Theorem~7]{HajlaszKoskelaTuominen2008}.
We recall that the Lipschitz constant $\lip(\cdot)$ was defined in \eqref{eq:lipschitz_constant}.
\begin{proof}[Proof of \cref{lem:LUC_implies_C1_embeds_Lip}]
	Let $x,y \in \Intr(\Omega)$ with $\abs{y - x} > r$ be given.
	Then,
	\begin{equation*}
		\abs{\varphi(y) - \varphi(x)}
		\le
		2 \cnorm{\varphi}
		\le
		\frac{2}{r} \cnorm{\varphi} \abs{y - x}.
	\end{equation*}
	On the other hand, if $x,y \in \Intr(\Omega)$
	satisfy $\abs{y - x}$,
	we find $\gamma \in C([0,1]; \Intr(\Omega))$
	with $\gamma(0) = x$, $\gamma(1) = y$
	and the Lipschitz constant of $\gamma$ is bounded by $C \abs{y - x}$.
	Owing to \eqref{eq:HDI},
	we have
	\begin{equation*}
		\abs{\varphi(y) - \varphi(x)}
		=
		\abs*{\int_0^1 \nabla\varphi(\gamma(t))^\top \gamma'(t) \de t }
		\le
		C \cdnorm{\nabla\varphi} \abs{y - x}.
	\end{equation*}
	This shows
	\begin{equation*}
		\abs{\varphi(y) - \varphi(x)}
		\le
		\max\set*{
			\frac{2}{r} \cnorm{\varphi},
			C \cdnorm{\nabla\varphi}
		}
		\abs{y - x}
		\qquad\forall x,y \in \Intr(\Omega).
	\end{equation*}
	Since $\varphi$ is continuous
	and since $\Omega$ is assumed to be the closure of $\Intr(\Omega)$,
	a limiting argument shows that the same inequality
	holds for all $x,y \in \Omega$.
\end{proof}

Next, we address the uniform Taylor expansion.
\begin{proof}[Proof of \cref{lem:quasiconvex_taylor}]
	Let $\varepsilon > 0$ be arbitrary.
	Let $r > 0$ and $C \ge 1$
	be the constants from \cref{def:LUQ}.
	Since $\nabla\varphi$ is uniformly continuous,
	there exists $D > 0$ such that
	\begin{equation*}
		\abs*{ \nabla \varphi(x) - \nabla \varphi(y) }
		\le
		\frac{\varepsilon}{C}
		\qquad
		\forall x, y \in \Omega, \abs{y - x} \le D.
	\end{equation*}
	We set $\delta := \min\set{D / C, r} / 2$.
	Now, let $x,y \in \Intr(\Omega)$ with $\abs{y - x} \le 2 \delta$
	be given.
	Since $\Intr(\Omega)$
	is uniformly locally quasiconvex,
	there exists
	$\gamma \in C([0,1]; \Intr(\Omega))$
	with
	$\gamma(0) = x$, $\gamma(1) = y$
	and Lipschitz constant at most $C \abs{y - x}$.
	Note that $C \abs{y - x} \le D$.
	Consequently, the range of $\gamma$ belongs to $\Intr(\Omega) \cap B_D(x)$.
	Using \eqref{eq:HDI},
	we get
	\begin{equation*}
		\abs*{
			\varphi(y) - \varphi(x) - \nabla \varphi(x)^\top (y - x)
		}
		=
		\abs*{ \int_0^1 \parens*{ \nabla\varphi(\gamma(t)) - \nabla\varphi(x) }^\top \gamma'(t) \de t}
		\le
		\varepsilon \abs{y - x}.
	\end{equation*}
	This inequality
	holds for all $x,y \in \Intr(\Omega)$
	with $\abs{y - x} \le 2 \delta$.
	Using a limiting argument yields the claim.
\end{proof}

Finally, we provide a counterexample
which shows that the conclusions
of \cref{lem:LUC_implies_C1_embeds_Lip,lem:quasiconvex_taylor}
may fail if $\Intr(\Omega)$
is not uniformly locally quasiconvex.
\begin{example}
	\label{ex:cantor}
	We define
	\begin{equation*}
		U
		:=
		\bigcup_{n = 1}^\infty \parens*{ \frac1{3^n}, \frac2{3^n} }
		\subset
		\R
		\quad\text{and}\quad
		\Omega := \cl U
		=
		\set{0} \cup \bigcup_{n = 1}^\infty \bracks*{ \frac1{3^n}, \frac2{3^n} }
		\subset
		\R
		.
	\end{equation*}
	Note that $U$ is the interior of $\Omega$,
	i.e., $\Omega$ is the closure of its interior, as required.
	Further, let $c \colon [0,1] \to [0,1]$
	be the (continuous) Cantor function and $f$ its restriction to $\Omega$,
	i.e.,
	$f(0) = 0$
	and
	$f(x) = 1/2^n$
	for all $x \in [1/3^n, 2/3^n]$.
	Obviously, $f$ is continuous.

	On each connected subset of $U$, the function $f$ is constant.
	Consequently, $f$ is differentiable on $U$
	with derivative $0$.
	Since the derivative can be continuously extended to $\Omega$,
	we have $f \in C^1(\Omega)$.

	However, $f$ is not Lipschitz continuous on $\Omega$
	since
	\begin{equation*}
		\frac{\abs{f(1/3^n) - f(0)}}{\abs{1/3^n - 0}}
		=
		\frac{3^n}{2^n}
		\to
		\infty
		\quad\text{as } n \to \infty.
	\end{equation*}
	Similarly, the Taylor expansion
	\begin{equation*}
		f(y)
		=
		f(0) + f'(0) (y - 0) + o(\abs{y})
		\qquad
		\text{as } \Omega \ni y \to 0
	\end{equation*}
	fails.

	Note that $\Omega$ is not connected.
	However, it is quite easy to modify the example to obtain connectedness.
	We set
	$\hat\Omega := \Omega \times [-1,0] \cup [0,1]^2$
	which is the closure of
	$\hat U :=  U \times (-1,0] \cup (0,1)^2$.
	We define the function
	$\hat f \colon \hat\Omega \to \R$
	via
	\begin{equation*}
		\hat f(x) :=
		\begin{cases}
			0 & \text{if } x_2 \ge 0, \\
			f(x_1) x_2^2 & \text{if } x_2 < 0.
		\end{cases}
	\end{equation*}
	It is straightforward to check that $\hat f \in C^1(\hat\Omega)$,
	but it fails to be Lipschitz continuous
	and the Taylor expansion fails at $(0,-1) \in \Omega$.
\end{example}

\clearpage
%%fakesection: Bibliography
\printbibliography
\end{document}